\RequirePackage[l2tabu, orthodox]{nag}

\documentclass[reqno]{amsart}

\usepackage[fontsize=11pt]{scrextend}

\usepackage{lmodern}
\usepackage[T1]{fontenc}
\usepackage[utf8]{inputenc}
\usepackage[english]{babel}
\usepackage[margin=2.8cm]{geometry}

\usepackage{microtype} %Better font spacing
\usepackage{comment}
\usepackage{varwidth} %for tableaux

\usepackage{amsmath,amssymb,amsthm,mathrsfs,listings,xskak,
latexsym,mathtools,mathdots,booktabs,enumitem,bm,url}
\usepackage[centertableaux]{ytableau}
\usepackage[pdftex]{hyperref}
\usepackage[capitalize,noabbrev]{cleveref}
\usepackage{caption}
\usepackage{subcaption}

\usepackage{todonotes}
\presetkeys%
    {todonotes}%
    {inline,backgroundcolor=yellow}{}

\newtheorem{theorem}{Theorem}[section]
\newtheorem{proposition}[theorem]{Proposition}
\newtheorem{lemma}[theorem]{Lemma}
\newtheorem{corollary}[theorem]{Corollary}
\newtheorem{conjecture}[theorem]{Conjecture}
\newtheorem{fact}[theorem]{Fact}

\theoremstyle{definition}
\newtheorem{definition}[theorem]{Definition}
\newtheorem{example}[theorem]{Example}
\newtheorem{remark}[theorem]{Remark}

\definecolor{lightblue}{rgb}{0.8,0.8,1.0}
\definecolor{lightgreen}{rgb}{0.8,1.0,0.8}
\definecolor{pBlue}{RGB}{86,139,190}
\definecolor{pCyan}{RGB}{149,186,201}
\definecolor{pSand}{RGB}{184,166,121}
\definecolor{pAlgae}{RGB}{87,115,135}
\definecolor{pSkin}{RGB}{236,216,167}
\definecolor{pGray}{RGB}{156,175,156}
\definecolor{pPink}{RGB}{215,114,127}
\definecolor{pOrange}{RGB}{211,153,80}

  \def\secondcircle{(210:1.75cm) circle (3.0cm)}
  \def\thirdcircle{(330:1.75cm) circle (3.0cm)}

  \tikzset{
dot/.style = {circle, fill, minimum size=#1,
              inner sep=0pt, outer sep=0pt},
dot/.default = 6pt % size of the circle diameter 
}

\newcommand{\defin}[1]{%
\relax\ifmmode%
\textcolor{blue}{#1}%
\else\textcolor{blue}{\emph{#1}}%
\fi%
}

\newcommand{\thsup}{\textnormal{th}}

\renewcommand{\rook}{\raisebox{-0.1em}{\symrook}}

\newcommand{\NN}{\text{NN}}

\newcommand*\circled[1]{\tikz[baseline=(char.base)]{\node[shape=circle,draw,inner sep=2pt] (char) {#1};}}

\tikzset{every picture/.append
	style={
		scale=1,
		x=1em,
		y=1em,
		entries/.style={xshift=-0.5em,yshift=-0.5em,font=\small},
		thickLine/.style={line width=1.4pt,line join=round},
		bgEntry/.style={xshift=-0.5em,yshift=-0.5em,
			regular polygon,regular polygon sides=4,fill,inner sep=0pt,minimum size=1.35em
		}
	}
}
\usetikzlibrary{calc}%for coordinate calculation in tikz
\DeclareMathOperator{\des}{des}

\newcommand{\pathMat}{\mathcal{P}}
\newcommand{\rookMat}{\mathcal{R}}
\newcommand{\precdot}{\prec\mathrel{\mkern-5mu}\mathrel{\cdot}}

\makeatletter
\newcommand{\preceqdot}{\mathrel{\mathpalette\pr@ceqd@t\relax}}
\newcommand{\pr@ceqd@t}[2]{%
  \begingroup
  \sbox\z@{$#1\prec$}\sbox\tw@{$#1\preceq$}%
  \dimen@=\dimexpr\ht\tw@-\ht\z@\relax
  {\preceq}%
  \mkern-5mu
  \raisebox{\dimen@}{$\m@th#1\cdot$}%
  \endgroup
}
\makeatother

\newcommand{\xvec}{\mathbf{x}}
\newcommand{\yvec}{\mathbf{y}}

\title{Rook matroids and log-concavity of $P$-Eulerian polynomials}

\author[Per Alexandersson]{Per Alexandersson}
\address{Department of Mathematics, Stockholm University, SE-106 91 Stockholm, Sweden}
\email{per.w.alexandersson@gmail.com}

\author[Aryaman Jal]{Aryaman Jal}
\address{Fakultät für Mathematik, Ruhr-Universität Bochum, 44801 Bochum, Germany}
\email{aryaman.jal@rub.de}


\begin{document}

\begin{abstract}
    We define and study rook matroids, the bases of which correspond to non-nesting rook placements on a skew Ferrers board. We show that rook matroids are a subclass of both transversal matroids and positroids; they also bear a subtle relationship to lattice path matroids that centers around not having the quaternary matroid $Q_{6}$ as a minor. The enumerative and distributional properties of non-nesting rook placements stand in contrast to those of usual rook placements: the non-nesting rook polynomial is not real-rooted in general, and is instead ultra-log-concave. We leverage this property together with a correspondence between rook placements and linear extensions of a poset to show that if $P$ is a naturally labeled width two poset, then the $P$-Eulerian polynomial $W_{P}$ is ultra-log-concave. This takes an important step towards resolving a log-concavity conjecture of Brenti (1989) and completes the story of the Neggers--Stanley conjecture for naturally labeled width two posets. 
\end{abstract}

\maketitle

%Now the doc is large, TOC is nice
\tableofcontents
\newpage

\section{Introduction}

\subsection{Overview}
Matroids can be defined from a wide variety of combinatorial objects. For example, the spanning trees of a graph, transversals of a set system, and dimers of a plabic graph respectively give rise to graphic matroids, transversal matroids~\cite{EdmondsFulkerson1965TransversalMatroids}, and positroids~\cite{Postnikov2006TotalPositivity}. A particularly well-behaved subclass of both transversal matroids and positroids is formed by the lattice path matroids introduced by Bonin, de Mier, and Noy~\cite{Bonin2003lattice, Bonin2006lattice}. The bases of a lattice path matroid $\pathMat_{\lambda/\mu}$ can be viewed as lattice paths contained within some skew shape $\lambda/\mu$. These matroids are known to have rich combinatorics; see~\cite{Bonin2003lattice} for a connection to the so-called tennis ball problem. They also play an important role in the algebro-combinatorial study of valuative invariants of matroids~\cite{Hampe2017IntersectionRingMatroids, Ferroni2023SchubertMatroidsDelannoyPaths, FerroniSchroter2022Valuative}.

In this paper, we introduce a new class of matroids, also parametrized by skew shapes, that shares many of the properties of lattice path matroids while being distinct from them. Given a skew shape $\lambda/\mu$, the bases of the \emph{rook matroid} $\rookMat_{\lambda/\mu}$ correspond to \emph{non-nesting} rook placements on $\lambda/\mu$, i.e., rook placements such that no rook lies South-East of another. We prove that, analogous to lattice path matroids, rook matroids are transversal and positroids; they are also closed under taking duals and direct sums. Moreover, the lattice path matroid and rook matroid on the same skew shape have the same Tutte polynomial, and hence the same number of bases, independent sets, spanning sets, and other invariants obtained as numerical specializations of the Tutte polynomial. However, despite sharing many of the same invariants, the two matroids are \emph{not} isomorphic in general; one contribution of this work is to identify precisely when an isomorphism exists and construct one explicitly. This hinges on whether or not $\lambda/\mu$ contains $332/1$ as a subshape; equivalently, whether or not the rook matroid contains the quaternary matroid $Q_{6}$ as a minor.    

Changing direction, we then apply the matroidal structure of non-nesting rook placements to prove that for a naturally labeled poset of width two, the corresponding $P$-Eulerian polynomial is ultra-log-concave, thereby making progress on a conjecture of Brenti~\cite{Brenti1989UnimodalAMS}. The context of this result is the Neggers--Stanley conjecture~\cite{Neggers1978, Stanley1986}, which asserted that $W_{P, \omega}$, the $(P, \omega)$-Eulerian polynomial of a labeled poset,   is real-rooted. This was disproved by Br\"{a}nd\'{e}n~\cite{Branden2004NeggersStanley} and Stembridge \cite{Stembridge2007NeggersStanleyCounter} in the early 2000s, the latter of whom provided a naturally labeled counterexample of width two. Despite these developments, the question of what weaker distributional property --- for example, log-concavity or unimodality --- holds for $W_{P, \omega}$ has remained open, with some progress in the graded, naturally labeled case~\cite{ReinerWelker2005CharneyDavis, Branden2008ActionsOnPermutations}. By deducing the strongest possible distributional property for $W_{P}$, our results thus complete the picture of the Neggers--Stanley conjecture in the naturally labeled, width two case. We do so by exhibiting a bijection between rook placements and linear extensions to show that a certain multivariate analog of $W_{P}$ is a Lorentzian polynomial in the sense of Br\"{a}nd\'{e}n and Huh~\cite{BrandenHuh2020Lorentzian}.    

\subsection{Outline and main results} 
In Section~\ref{section:prelims}, we gather background material on rook placements, distributional properties of combinatorial sequences, and matroid theory. The following is a summary of our main contributions together with the organization of the rest of the paper:
\begin{enumerate}
\item In Section~\ref{section:rook_matroid}, we define the rook matroid $\rookMat_{\lambda/\mu}$ (Theorem~\ref{thm:rook_is_transversal}) via the transversal description and study some of its properties. We show that rook matroids are positroids (Theorem~\ref{thm:rook_matroids_positroids}), give a precise criterion on $\lambda/\mu$ for when the rook matroid is isomorphic to the corresponding lattice path matroid $\pathMat_{\lambda/\mu}$ (Theorem~\ref{thm:rookPathBij}), and prove that $\rookMat_{\lambda/\mu}$ and $\pathMat_{\lambda/\mu}$ are nevertheless Tutte-equivalent in all cases (Theorem~\ref{thm:sameTutte}). 

\item In Section~\ref{section:distributional_properties}, we show that the generating polynomial $M_{\lambda/\mu}(t)$ of the size of non-nesting rook placements on a skew shape $\lambda/\mu$ is ultra-log-concave and not real-rooted in general. We give the first example of a lattice path matroid that \emph{does not} have the half-plane property (Theorem \ref{thm:LPM_not_HPP}), answering a question that was raised in~\cite{ChoeOxleySokalWagner2004} regarding transversal matroids with the half-plane property. 

\item In Section~\ref{section:rooks_as_linear_extensions}, we derive a bijective correspondence between skew shapes $\lambda/\mu$ and width two posets $P$ that maps non-nesting rook placements on $\lambda/\mu$ to linear extensions (Theorem~\ref{thm:path_to_poset}) of $P$. Consequently, the $P$-Eulerian polynomial can be expressed as $M_{\lambda/\mu}$, settling the  ultra-log-concavity consequence of the Neggers--Stanley conjecture for this class of posets (Corollary~\ref{neggers_stanley_width2}). We end by describing precisely when $M_{\lambda/\mu}(t)$ is gamma-positive.
\end{enumerate}

\section{Preliminaries}\label{section:prelims}
In this section, we define the basic combinatorial objects that will be the focus of our 
interest: non-nesting rook placements on boards and their generating polynomials. 
We will also recall the necessary background on matroids and notions from the geometry 
of polynomials that will be used in later sections.

\subsection{Boards, skew shapes and rook placements}\label{subsection:boards_skew-shapes_rooks}

We use $[n]$ to mean the set $\{1,2, \dotsc , n\}$ and $[m, n]$ to mean the 
discrete interval $\{m, m+1, \dotsc , n\}$. 
The set of size-$k$ subsets of $[n]$ is denoted by $\binom{[n]}{k}$.

A \defin{board} $B$ with $r$ rows and $c$ columns is a subset of the rectangular grid $[r] \times [r+1, r+c]$. We label the rows with elements from $[r]$ and the columns with elements from $[r+1, r+c]$ 
respectively, as in Example~\ref{example:non_nesting_rook_board}. 
The row labeling is done from the top to bottom and the column labeling is done from left to right. 
We refer to cells $(i, j)$ of $B$ with respect to this labeling unless otherwise specified. Given a board $B$ with $r$ rows and $c$ columns, the \defin{bipartite graph corresponding to $B$} is the graph $\Gamma_{B} = ([1, r]\sqcup [r+1,r+c], E)$ where $i \in [1,r]$ is connected to $j \in [r+1, r+c]$ if $(i, j) \in B$.
For most of what follows, we restrict ourselves to special classes of boards coming from integer partitions.  

An \defin{integer partition} $\lambda = (\lambda_{1}, \ldots , \lambda_{n})$ is any non-increasing tuple of positive integers. 
Following the English convention, the \defin{Ferrers diagram} of $\lambda$ is a left-justified, weakly 
decreasing array of boxes. 
Two special partitions correspond to Ferrers diagrams in the shape of a staircase and rectangle, respectively. 
For the former, $\delta_{n}$ denotes the partition $(n, n-1, \ldots , 1)$ and for 
the latter, $(a)^{b}$ denotes the partition consisting of $b$ parts each of size $a$. 

Given two integer partitions $\lambda \supseteq \mu$ of length at most $r$,
we let the \defin{skew Ferrers board} associated with $\lambda/\mu$ be the 
board with cells
\[
  \left\{ (i, r+j) : 1 \leq i \leq r \text{ and } \mu_i <  j \leq \lambda_i \right\}.
\]
When $\mu$ is the empty partition, we treat $\mu$ as the all-zeros vector. 
In what follows, we also refer to a skew Ferrers board as a \defin{skew shape}. We call a skew shape \defin{non-degenerate} if for all $i$, $\lambda_{i} > \mu_{i}$ and every column of the bounding rectangle contains at least one cell of $\lambda/\mu$. A \defin{subshape} of a skew shape $\lambda /\mu$ is the skew shape obtained by deleting rows or columns of $\lambda/\mu$. The \defin{size} of a skew shape $\lambda/\mu$ is the number of boxes in its diagram and is denoted $|\lambda/\mu|$. An \defin{outer corner} of $\lambda/\mu$ is a cell $(i, j) \notin \lambda/\mu$ such that $(i, j-1), (i-1, j) \in \lambda/\mu$. An \defin{inner corner} of $\lambda/\mu$ is a cell $(i, j) \notin \lambda/\mu$ such that $(i+1, j), (i, j+1) \in \lambda/\mu$. 

By a \defin{non-attacking} placement of rooks on a board, we mean a configuration of rooks where no two rooks
can take each other in the sense of chess. More formally, a non-attacking rook 
placement $\rho$ on a board $B$ is a subset of $B$, no two cells of which share a row or column index. Two rooks in a non-attacking rook placement form a \defin{nesting} if 
one rook lies South-East of another. Hereafter, a \defin{non-nesting rook placement} 
(or just rook placement, when it is clear that the setting is a non-nesting one) on a 
board $B$ is a non-attacking rook placement such that no pair of rooks forms a nesting. It is straightforward to see that a non-attacking rook placement on $B$ corresponds to a matching of the graph $\Gamma_{B}$. 

We will frequently move between the settings of non-nesting rook placements and lattice paths in a skew shape. By such a lattice path, we mean a path starting from the Southwest corner of $\lambda/\mu$ and terminating at the Northeast corner of $\lambda/\mu$ that consists only of North and East steps. Examples can be found in Section~\ref{subsection:lattice_path_matroids}.

\begin{example}\label{example:non_nesting_rook_board}
From left to right below is a board, a nesting rook placement and a non-nesting rook placement. 
\[
\ytableausetup{boxsize=1.3em}
\begin{ytableau}
\none & \none[\scriptstyle{4}] & \none[\scriptstyle{5}] & \none[\scriptstyle{6}] \\
 \none[\scriptstyle{1}] & \none  & \none  &  \;  \\
 \none[\scriptstyle{2}] &  \none & \;  & \none   \\
  \none[\scriptstyle{3}] & \; & \none & \;
\end{ytableau}
\qquad 
\qquad 
\begin{ytableau}
\none & \none[\scriptstyle{4}] & \none[\scriptstyle{5}] & \none[\scriptstyle{6}] \\
 \none[\scriptstyle{1}] & \none  & \none  &  \;  \\
 \none[\scriptstyle{2}] &  \none & \rook \;  & \none   \\
  \none[\scriptstyle{3}] & \; & \none & \rook \;
\end{ytableau}
\qquad 
\qquad 
\begin{ytableau}
\none & \none[\scriptstyle{4}] & \none[\scriptstyle{5}] & \none[\scriptstyle{6}] \\
 \none[\scriptstyle{1}] & \none  & \none  & \rook  \;  \\
 \none[\scriptstyle{2}] &  \none & \rook \;  & \none   \\
  \none[\scriptstyle{3}] & \rook \; & \none & \;
\end{ytableau}
\]
\end{example}

We denote the set of non-nesting rook placements on a skew 
shape $\lambda/\mu$ by $\defin{\NN_{\lambda /\mu}}$ and define the generating polynomial of its $k$-subsets as follows.

\begin{definition}\label{defn:non_nesting_rook_polynomial}
Given a skew shape $\lambda/\mu$ with $r$ rows, the \defin{non-nesting rook polynomial} of $\lambda /\mu$ is defined by 
\[
M_{\lambda /\mu}(t) = \sum_{k=0}^{r}r_{k}(\lambda /\mu)t^{k},
\]
where $r_{k}(\lambda /\mu)$ is the number of non-nesting rook placements on $\lambda /\mu$ of size $k$. 
\end{definition}

Note that the degree of $M_{\lambda/\mu}$ equals the length of the longest Northeast to Southwest diagonal that fits inside $\lambda /\mu$, which is not necessarily equal to $r$.  

\begin{example}\label{eg:rectangles}
Let $\lambda_{a, b} = (a^{b})$. 
Then we have 
\[
M_{\lambda_{a, b}}(t) = \sum_{k=0}^{\min(a, b)}\binom{a}{k}\binom{b}{k}t^{k}.
\]

This can be seen directly: choose a set $\{j_{1}< \ldots < j_{k}\}$ of columns in $\binom{a}{k}$ 
ways and choose a set of rows $\{i_{1} > \ldots > i_{k}\}$ independently in $\binom{b}{k}$ ways. 
Form the rook placement with cells $\{(i_{\ell}, j_{\ell}): \ell=1, \ldots k\}$; by virtue of 
the ordering, this must be non-nesting. 

By specializing $a=b=n$, we obtain the type B Narayana polynomial: 
\[
W(t) = \sum_{k=0}^{n}\binom{n}{k}^{2}t^{k}.
\]
\end{example}

\begin{example}[Narayana polynomials]\label{eg:Narayana_poly}
The non-nesting rook polynomial for the staircase 
Ferrers board $\lambda = \delta_{n} = (n,n-1,n-2,\dotsc,2,1)$ is
\[
M_{\delta_{n}}(t) = t^{-1} N_{n+1}(t),
\]
where $N_{n}(t) = \sum_{k=1}^{n}\frac{1}{n}\binom{n}{k}\binom{n}{k-1}t^{k}$ denotes the $n^\thsup$ Narayana polynomial. This connection with lattice paths is addressed in greater generality in Section~\ref{subsection:lattice_path_matroids}. 
\end{example}

Distributional properties and multivariate generalizations of $M_{\lambda/\mu}$, 
together with its realizability as a $P$-Eulerian polynomial of a poset,  will be studied in later sections. 

\subsection{Matroids}\label{subsection:prelim_matroids}
We use the following axiomatization of matroids in terms of its bases; for any undefined terms 
hereafter we refer the reader to \cite{Oxley2011MatroidBook}.
\begin{definition}\label{defn:matroid_defn}
    A \defin{matroid} $M$ is a pair $(E, \mathcal{B})$ where $E$ is a finite set and $\mathcal{B}$ is a collection of subsets of $E$, called \defin{bases}, satisfying the following two properties: \begin{enumerate}
        \item $\mathcal{B} \neq \emptyset$.
        \item For each distinct pair $B_{1}, B_{2}$ in $\mathcal{B}$ and element $a \in B_{1} \setminus B_{2}$, there exists $b \in B_{2} \setminus B_{1}$ such that $(B_{1} \setminus \{a\}) \cup \{b\} \in \mathcal{B}$.
    \end{enumerate}
\end{definition}

The second property above is referred to as the \defin{basis-exchange axiom} for matroids. Numerous other axiomatizations of matroids exist; some can be found in \cite[Chapter 1]{Oxley2011MatroidBook}. One such axiomatization can be made in terms of independent sets. A set $I \subset E$ is an \defin{independent set} of a matroid $M = (E, \mathcal{B})$ if $I \subseteq B$, 
for some $B \in \mathcal{B}$. 

An important example of matroids is the \defin{uniform matroid} $U_{k, n}$, 
defined as the  matroid with ground set $[n]$ and bases given by all $k$-subsets of $[n]$. We recount more terminology associated to matroids. 
\begin{enumerate}
\item A \defin{loop} of a matroid $M$ is an element of the ground set not contained in any basis of $M$. A \defin{coloop} is an element of the ground set that is contained in every basis of $M$.
    \item Given a matroid $M = (E, \mathcal{B})$ and an element $e \in E$, the \defin{deletion} and \defin{contraction} of $M$ by $e$ are the matroids with ground set $E \setminus e$ that are respectively denoted by $M\setminus e$ and $M /e$ and the bases of which are respectively given by: \begin{align*}
    \mathcal{B}(M \setminus e) &= \begin{cases}
        \{B \in \mathcal{B}: e \notin B\} \quad &\text{if $e$ is not a coloop of $M$}\\
        \{B \setminus e: B \in \mathcal{B} \} \quad &\text{if $e$ is a coloop of $M$}
    \end{cases}\\ 
    \mathcal{B}(M / e) &= \begin{cases}
    \{B\setminus e : B \in \mathcal{B}, e \in B\} \quad &\text{if $e$ is not a loop of $M$}\\
    \mathcal{B}(M\setminus e) \quad &\text{if $e$ is a loop of $M$}   
    \end{cases}
\end{align*}

A matroid obtained from $M$ by a sequence of deletions and contractions is called a \defin{minor} of $M$. A class $\mathcal{C}$ of matroids is \defin{minor-closed} if given a matroid $M \in \mathcal{C}$, every minor of $M$ is also in $\mathcal{C}$. 

\item The \defin{dual} of a matroid $M = (E, \mathcal{B})$ is the matroid $M^{*}$ on the ground set $E$ and bases
$\mathcal{B}(M^{*}) = \{E\setminus B: B \in \mathcal{B}\}$.
\end{enumerate}

 Transversals of set systems give rise to an important class of matroids. 
 Given a set system $\mathcal{N} = (N_{1}, \ldots , N_{r})$ of not necessarily distinct 
 subsets $N_{i}$ of some ground set $E$, a \defin{transversal} $X$ of $\mathcal{N}$ is a 
 subset of $E$ consisting of $r$ distinct elements, one from each set in $\mathcal{N}$: that is, $X = \{a_{i}: i=1, \ldots , r\}$ such that $a_{i} \in N_{i}$ for $i=1, \ldots , r$. 
 A \defin{partial transversal} of $\mathcal{N}$ is a transversal for 
 some subsystem $\mathcal{M} = (N_{i})_{i \in J}$ indexed by a subset $J \subset [r]$. 
 The following theorem due to Edmonds and Fulkerson laid the foundation for the theory of transversal matroids:

\begin{theorem}\cite{EdmondsFulkerson1965TransversalMatroids}
    The partial transversals of $\mathcal{N}$ are the independent sets of a matroid on $E$.
\end{theorem}

Matroids that arise in this way are called \defin{transversal matroids} and $\mathcal{N} = (N_{1}, \ldots , N_{r})$ is called a \defin{presentation} of the matroid. 
The bases of a transversal matroid are given by the maximal partial transversals of $\mathcal{N}$. When $\mathcal{N}$ admits a transversal, the maximal partial transversals are precisely the transversals of $\mathcal{N}$. 

\subsection{Real-rootedness, log-concavity and unimodality}
In this subsection, we recall some basic notions in the geometry of polynomials.

Given a sequence of non-negative integers $(a_{k})_{k =0}^{n}$, it is convenient to work instead with its generating polynomial $A(t) =  \sum_{k=0}^{n}a_{k}t^{k}$.  The following notions, listed in descending order of strength, encode distributional properties of the sequence: \begin{enumerate}
    \item The generating polynomial $A(t)$ is \defin{real-rooted} if $A$ has only real roots as univariate polynomial. Note that by the non-negativity of the $a_{k}$, the roots of $A$ will necessarily be non-positive.
    \item The sequence $(a_{k})_{k =0}^{n}$ is \defin{ultra-log-concave} if \[\left(\dfrac{a_{k}}{\binom{n}{k}}\right)^{2} \geq \dfrac{a_{k-1}}{\binom{n}{k-1}}\cdot \dfrac{a_{k+1}}{\binom{n}{k+1}} \quad \text{for all $1\leq k \leq n-1$.}
    \]
    \item The sequence $(a_{k})_{k =0}^{n}$ is \defin{log-concave} if \[a_{k}^{2} \geq a_{k-1} \cdot a_{k+1} \quad \text{for all $1\leq k \leq n-1$.}
    \]
    \item The sequence $(a_{k})_{k =0}^{n}$ is \defin{unimodal} if there exists some $0 \leq m \leq n$ such that:
    \[a_{0}\leq a_{1} \leq \ldots \leq a_{m-1 }\leq a_{m} \geq a_{m+1} \geq \ldots \geq a_{n}.
    \]
\end{enumerate}

The sequence $(a_{k})_{k =0}^{n}$ 
has \defin{no internal zeros} if 
\[
a_{i}a_{j} > 0 \implies a_{k}>0 \quad \text{for all $0\leq i < j < k \leq n$}. 
\]

As mentioned above, the following implications hold: $(1) \implies (2) \implies (3) \implies (4)$. 
Here $(3)$ implies $(4)$ only if $(a_{k})_{k=0}^{n}$ has no internal zeros.
For a comprehensive and modern survey of these notions, and some of their multivariate counterparts that are addressed in later sections of this paper, see \cite{Branden2015}.

\section{The rook matroid}\label{section:rook_matroid}
In this section, we show that there is a natural matroid structure underpinning the set of non-nesting rook placements on a skew Ferrers board. The associated class of matroids is a dual-closed subclass of both transversal matroids and positroids. However, it is not minor-closed and therein lies the subtlety in the relationship between rook matroids and lattice path matroids. The schematic in Figure~\ref{fig:matroid_venn} summarizes the various containment and characterization properties that rook matroids satisfy with respect to other known classes of matroids; we will prove these in Sections~\ref{subsection:lattice_path_matroids} and \ref{subsection:positroids}. 

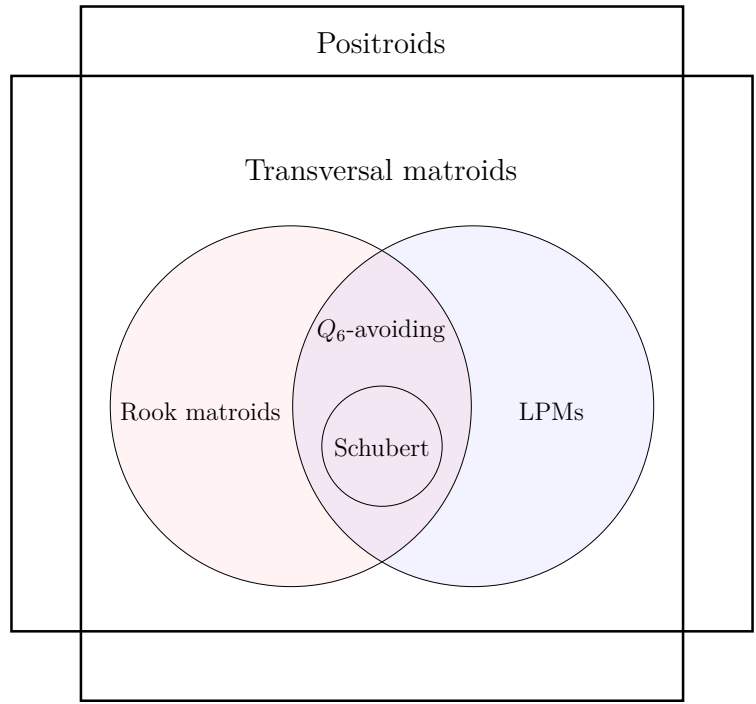
\begin{figure}[!htbp]
    \centering
    \scalebox{0.8}{
   \begin{tikzpicture}
    \draw[black, very thick] (-13,-15) rectangle (13,15);
     \draw[black, very thick] (-16,-12) rectangle (16,12);
      \begin{scope}
      \fill[red,opacity=0.05]
     \secondcircle

    \thirdcircle;
      \end{scope}
      \begin{scope}

   \fill[blue,opacity=0.05]
    \thirdcircle;
      \end{scope}
     
      \draw \secondcircle node [text=black,label={[xshift=-1.5cm, yshift=-0.5cm]{\large Rook matroids}}, label={[xshift=1.5cm, yshift=0.8cm]{\large $Q_{6}$-avoiding}}, label={[xshift=1.5cm, yshift=3.5cm]{\Large Transversal matroids}}, label={[xshift=1.5cm, yshift=5.6cm]{\Large Positroids}}] {};
      \draw \thirdcircle node [text=black, label={[xshift=1.3cm, yshift=-0.5cm]{\large LPMs}}] {};
      %\path[black, bend right] ($(330:1.75)-(3,0)$) edge ($(210:1.75)+(3,0)$); ;
      \node[circle, draw = black, minimum size = 2 cm, label={[xshift=0cm, yshift=-1.3cm]{\large Schubert}}] (A) at (0,-4) {};
    \end{tikzpicture}
    }
    \caption{Various matroid subclasses.}
    \label{fig:matroid_venn}
\end{figure}

We begin by studying the structural properties of the rook matroid that arise from its transversal description.

\subsection{Structural results}\label{subsection:structural_results}

Let $\lambda/\mu$ be a skew shape. Throughout we draw and refer to Young diagrams $\lambda$ and $\mu$ in terms of the English convention, that is as a left-justified array of boxes. With this convention in place, we index
the rows of $\lambda/\mu$ by $\{1,2,\dotsc,r\}$ and the columns by $\{r+1,\dotsc,r+c\}$. We set $E_{\lambda/\mu} = [r+c]$. 

Given a non-nesting rook placement $\rho$ on $\lambda/\mu$,
we associate to it the set $R(\rho) \cup C(\rho)$
where $R(\rho)$ is the set of row indices that are not occupied by $\rho$ and $C(\rho)$ is the set of column indices that are occupied by $\rho.$ Note that for each non-nesting rook placement $\rho$, we have that $|R(\rho) \cup C(\rho)| = r$, the number of rows of the shape. We gather these sets into a single collection: 
\[
 \defin{\rookMat_{\lambda/\mu}} \coloneqq \{ R(\rho) \cup C(\rho) : \rho  \in \NN_{\lambda/\mu} \},
\]
where as before $\NN_{\lambda/\mu}$ is the set of non-nesting rook placements on $\lambda/\mu$. We use the following elementary fact about diagrams characterizing skew shapes.

\begin{fact}\label{fact:skew_shape_characterization}
    A board $B$ is a skew shape if and only if for every pair of cells $(i, j), (k, \ell) \in B$ with $i<k$ and $j< \ell$, we also have $(k, j), (i, \ell) \in B$.
\end{fact}

Given a skew shape, we may define a set system as follows. For $i \in [r]$, 
let $A_{i}$ be the set of column indices occupied by row $i$ together with $i$, that is, \[A_{i} = \{j \in [r+ 1, r+c]: (i, j) \in \lambda/\mu \} \cup \{i\}.\]
Then $\defin{\mathcal{A}_{\lambda/\mu}} \coloneqq  (A_1,\dotsc,A_r)$
is a set system on $E_{\lambda / \mu}$ which defines a transversal matroid.

\begin{example}
Consider the skew shape $\lambda/\mu = 77553/42$.
\[
\ytableausetup{boxsize=1.3em}
 \begin{ytableau}
 \none & \none[6] &  \none[7] & \none[8] & \none[9] & \none[10] & \none[11] & \none[12] \\
\none[1] & \none & \none & \none & \none &  & & \\
 \none[2] &\none & \none &  &  &  & &  \\
\none[3] &  &  &  & &  \\
\none[4] &   &  &  & &  \\
\none[5] &    &  &  \\
 \end{ytableau}
 \qquad 
 \hspace{1cm}
  \begin{ytableau}
 \none & \none[6] &  \none[7] & \none[8] & \none[9] & \none[10] & \none[11] & \none[12] \\
\none[1] & \none & \none & \none & \none &  & \rook & \\
 \none[2] &\none & \none &  &  & \rook & &  \\
\none[3] &  &  &  & &  \\
\none[4] &   &  &  & &  \\
\none[5] & \rook   &  &  \\
 \end{ytableau}
\]
This skew shape has corresponding set system
\begin{align*}
A_1 &= \{1,10,11,12\} & A_2 &= \{2,8,9,10,11,12\}  & A_3 &= \{3,6,7,8,9,10\} \\
A_4 &= \{4,6,7,8,9,10\} & A_5 &= \{5,6,7,8\}.
\end{align*}
The non-nesting rook placement shown on the right above can be represented by $R(\rho)\cup C(\rho) = \{3,4,6,10,11\}$ where $11 \in A_1$, $10 \in A_2$, $3 \in A_3$, $4 \in A_4$, $6 \in A_5$. 
\end{example}

\begin{theorem}\label{thm:rook_is_transversal}
 The set $\rookMat_{\lambda/\mu}$ is
 the collection of bases of the transversal matroid with set system $\mathcal{A}_{\lambda / \mu}$.
\end{theorem}
\begin{proof}
Given $\rho \in 
\NN_{\lambda/\mu}$, it is clear that $R(\rho) \cup C(\rho)$ is a transversal of $\mathcal{A}_{\lambda/\mu}$. Indeed, if $\rho$ has no rook in row $i$, match $i$ to $A_{i}$; if $\rho$ has a rook in row $i$ at $r+j$, match $r+j$ to $A_{i}$. Since $\rho$ is non-attacking, the chosen elements are distinct.

Conversely, suppose we are given a transversal $T$ of $\mathcal{A}_{\lambda / \mu}$. Let $[r] \setminus T = \{i_{1}< \cdots < i_{k}\}$ and $T \cap [r+1, r+c] = \{r+j_{1}, \cdots , r + j_{k}\}$ where $j_{1}> \cdots > j_{k}$. We claim that $\rho = \{(i_{\ell}, r+j_{\ell}): \ell =1, \ldots , k\}$ is a non-nesting rook placement on $\lambda/\mu$. By the ordering on $i_{\ell}, j_{\ell}$, the rook placement $\rho$ is non-nesting; it remains to show that $\rho$ lies on $\lambda/\mu$, that is, $\mu_{i_{\ell}} < j_{\ell} \leq \lambda_{i_{\ell}}$ for $\ell = 1, \ldots , k$. For the first inequality, suppose there exists $s \in [k]$ such that $j_{s} \leq \mu_{i_{s}}$. Then, for $t \leq s $ and $a \geq s$, we have $\mu_{i_{t}} \geq \mu_{i_{s}} \geq j_{s} \geq j_{a}$. This implies that for indices $t \leq s$ and $a \geq s$ column $r + j_{a}$ does not have a cell in row $i_{t}$. Hence the $k-s+1$ column elements $r+j_{s}, r+j_{s+1}, \ldots , r+j_{k}$, would have to be chosen from the $k - s $ sets $A_{i_{s+1}}, \ldots , A_{i_{k}}$, which contradicts the fact that $T$ is a transversal. 

For the second inequality, suppose to the contrary that there exists $s \in [k]$ such that $j_{s} > \lambda_{i_{s}}$. Then, for $t \geq s$ and $a \leq s$, we have $\lambda_{i_{t}} \leq \lambda_{i_{s}} < j_{s} \leq j_{a}$. This implies that for indices $t$ and $a$ in the aforementioned ranges, column $r+j_{a}$ does not have a cell in row $i_{t}$. Hence the $s$ columns $r+j_{1}, \ldots , r+j_{s}$ would have to be chosen from the $s-1$ sets $A_{i_{1}}, \ldots , A_{i_{s-1}}$, which contradicts the fact that $T$ is a transversal. Thus, $\rho$ is a non-nesting rook placement on $\lambda/\mu$ with $R(\rho) \cup C(\rho) = T$, which completes the proof.
\end{proof}

Identifying the resulting matroid by its set of bases, we call $\rookMat_{\lambda / \mu}$ the \defin{rook matroid} on $\lambda / \mu$; for any matroid $M$, we say that $M$ is a rook matroid if it is isomorphic to a rook matroid of the form $\rookMat_{\lambda / \mu}$. The uniform matroid $U_{k, n}$ corresponds to the rook matroid on the $(n-k) \times k$ rectangle. Later on, we will see how rook matroids interact with other well-studied classes of matroids, including lattice path matroids and positroids.

Empty rows and columns have the expected matroidal interpretation. A column $j$ of $\lambda/\mu$ is empty if and only if the corresponding column label belongs to no set $A_i$, and hence is a loop of $R_{\lambda/\mu}$. Similarly, row $i$ is empty, equivalently $\lambda_i=\mu_i$, if and only if $A_i= \{i\}$, and hence the row label $i$ is a coloop. Thus deleting empty columns and rows from the skew shape amounts to removing loops and coloops from the corresponding rook matroid.

It is natural to ask if the skew shape assumption is necessary in the above theorem and the proposition below confirms that this is true.

\begin{proposition}\label{prop:skew_shape_necessary}
    Let $B$ be a board with $r$ rows and $c$ columns, and let \[
    \mathcal{B} = \{R(\sigma) \cup C(\sigma): \sigma \in \mathrm{NN}_{B}\},
    \]
    where $\mathrm{NN_{B}}$ denotes the set of non-nesting rook placements on $B$. If $M = ([r+c], \mathcal{B})$ is a matroid with bases $\mathcal{B}$, then the board $B$ must be equal to some skew shape. 
\end{proposition}

\begin{proof}
We begin by noting that if $i$ is a row index, then the bases of  $M/ i$ correspond to the non-nesting rook placements on $B$ such that row $i$ is not occupied. Similarly, if $j$ is a column index, then the bases of $M\setminus j$ correspond to the non-nesting rook placements on $B$ such that column $j$ is not occupied.

Now suppose $B$ is not a skew shape. Then, by Fact~\ref{fact:skew_shape_characterization}, there exists $(i, j), (k, \ell) \in B$ with $i<k$ and $j<\ell$ such that at least one of $(k, j)$ or $(i, \ell)$ do not lie in $B$. Let $M_{1} = (M\setminus Y)/X$ where $X = [r]\setminus \{i, k\}$ and $Y = [r+1, r+c] \setminus \{j, \ell\}$. Let $M_{2}$ be the matroid obtained by deleting all loops and coloops of $M_{1}$. Then, by the observation in the first paragraph of this proof, $M_{2}$ is isomorphic to the matroid $M_{3} = ([4], \mathcal{B}')$ where \[
\mathcal{B}' = \{R(\sigma) \cup C(\sigma): \sigma \in \mathrm{NN}_{B'}\},
\]
and $B'$ is one of the three boards in Figure~\ref{fig:boards_are_skew}. In each case, $\{2,3\}, \{1,4\} \in \mathcal{B}'$ and $3 \in \{2,3\}\setminus \{1,4\}$ but neither of $\{1,2\}$ nor $\{2,4\}$ are valid non-nesting rook placements on any of these boards. 
\end{proof}

\begin{figure}[!ht]
    \centering
    \begin{ytableau} \none & \none[3] & \none[4]  \\ \none[1] & \rook & & \none  \\ \none[2] & \none  & \\ \none & \none & \none[23]  \end{ytableau}
    \hspace{-0.5cm}
  \begin{ytableau} \none & \none[3] & \none[4]  \\ \none[1] &  & & \none  \\ \none[2] & \none  & \rook  \\ \none & \none & \none[14]  \end{ytableau}
\qquad
 \begin{ytableau} \none & \none[3] & \none[4]  \\ \none[1] & \rook  & \none  \\ \none[2] &  &  \\ \none & \none & \none[23]
   \end{ytableau}
   \hspace{-0.25cm}
\begin{ytableau} \none & \none[3] & \none[4]  \\ \none[1] &  & \none  \\ \none[2] &  & \rook  \\ \none & \none & \none[14]  \end{ytableau}
\qquad
\begin{ytableau} \none & \none[3] & \none[4]  \\ \none[1] & \rook  & \none  \\ \none[2] & \none  &  \\ \none & \none & \none[23]
   \end{ytableau}
   \hspace{-0.25cm}
\begin{ytableau} \none & \none[3] & \none[4]  \\ \none[1] &  & \none  \\ \none[2] & \none  & \rook  \\ \none & \none & \none[14]  \end{ytableau}

\caption{For each of the three boards $B'$ above, the collection $\mathcal{B}'$ in the proof of Proposition~\ref{prop:skew_shape_necessary} fails to satisfy the basis-exchange axiom.}
\label{fig:boards_are_skew}
\end{figure}

More can be said about the transversal structure of the rook matroid. A \defin{fundamental transversal matroid} is a matroid $M$ that has a basis $B$ (called a fundamental basis) such that each cyclic flat $F$ of $M$ is spanned by some subset of $B$. An equivalent characterization identifies fundamental transversal matroids as those transversal matroids that have a presentation $(A_{1}, \ldots , A_{r})$ such that for each $i \in [r]$ there is some element in $A_{i}$ that is in none of the other sets. This allows us to immediately note that rook matroids are fundamental transversal, as the distinguished element in each set of the presentation of $\rookMat_{\lambda/\mu}$ is simply the associated row label. The fundamental basis of the rook matroid then corresponds to the empty rook placement. See \cite{BoninKungdeMier2011FundamentalTransversal} for more about fundamental transversal matroids.

We now examine rook matroids through the structural properties of being closed under taking direct sums, duals, and minors. The first two properties hold  while the third fails.

\begin{lemma}\label{lem:direct_sum}
The class of rook matroids is closed under direct sums.
\end{lemma}
\begin{proof}
A direct sum of rook matroids corresponds to the rook matroid on the \defin{direct sum} of the corresponding skew shapes. That is, if $\lambda_{1} / \mu_{1}$ and $\lambda_{2} / \mu_{2}$ are two skew shapes, then the direct sum $\lambda/\mu \coloneqq (\lambda_{1} / \mu_{1})\oplus(\lambda_{2} / \mu_{2})$ is obtained by positioning the North-East corner of $\lambda_{1} / \mu_{1}$ to the immediate South-West of the South-West corner of $\lambda_{2} / \mu_{2}$; see Figure~\ref{fig:direct}. Then by the obvious isomorphism, $\rookMat_{\lambda/\mu} \cong \rookMat_{\lambda_{1} / \mu_{1}} \oplus \rookMat_{\lambda_{2} / \mu_{2}}$.
\end{proof}

\begin{figure}[!ht]
    \centering
   \begin{tikzpicture}[inner sep=0in,outer sep=0in]
\node (n) {\begin{varwidth}{6cm}{
\ytableausetup{boxsize=1.25em}
 \begin{ytableau} \none & \none[5] & \none[6] & \none[7] & \none[8] & \none[9] & \none[10] & \none[11]  \\ \none[1] & \none & \none & \none & \none & \none &  &  \\ \none[2] & \none & \none & \none & \none & \none &  &  \\ \none[3] & &  &  &  &  \\ \none[4] & & &  &  \\ \end{ytableau}}\end{varwidth}};
\end{tikzpicture}
    \caption{The direct sum $\lambda_{1}/\mu_{1} \oplus \lambda_{2}/\mu_{2}$ of the rook matroids on  $\lambda_{1}/\mu_{1} = 54$ and $\lambda_{2}/\mu_{2} = 22$ is the rook matroid on the direct sum of their shapes, $\lambda / \mu = 7754/55$.}
    \label{fig:direct}
\end{figure}
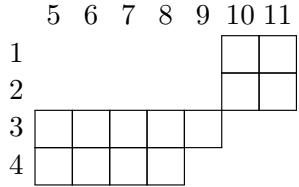

The \defin{conjugate} of a partition $\lambda$ is the partition $\lambda'$ obtained by taking the transpose of the Ferrers diagram of $\lambda$. Duals of rook matroids are simply rook matroids on conjugated skew shapes; the proof is straightforward and hence omitted. 

\begin{proposition}\label{prop:matroid_dual}
Let $\lambda / \mu$ be a skew shape and let $\lambda' / \mu '$ denote its conjugate. Then the following isomorphism holds: $\rookMat^*_{\lambda/\mu} \cong
\rookMat_{\lambda'/\mu'}$.
\end{proposition}

The following lemma shows that minors of rook matroids --- in two cases --- are 
straightforward to describe combinatorially.
We make use of this lemma several times throughout the paper; its proof follows from the definitions of deletion and contraction.

\begin{lemma}\label{lem:easyRookMinors}
Let $\lambda / \mu$ be a skew shape and $\rookMat_{\lambda/\mu}$ be the corresponding rook matroid.
Suppose $x$ is a column index and $y$ is a row index of $\lambda/\mu$. Then 
\begin{itemize}
    \item The deletion $\rookMat_{\lambda/\mu} \setminus x$ is isomorphic to
    the rook matroid obtained from the shape $\lambda/\mu$ with column $x$ removed;
    \item The contraction $\rookMat_{\lambda/\mu} / y$ is isomorphic to
    the rook matroid obtained from the shape $\lambda/\mu$ with row $y$ removed;
 \item The bases in the contraction $\rookMat_{\lambda/\mu} / x$ correspond to 
 non-nesting rook placements on $\lambda/\mu$ where column $x$ is occupied by a rook;
 \item The bases in the deletion $\rookMat_{\lambda/\mu} \setminus y$ correspond to 
 non-nesting rook placements on $\lambda/\mu$ where row $y$ is occupied by a rook.
\end{itemize}
\end{lemma}

Rook matroids can even be shown to be closed under contraction, a fact that can be deduced using the theory developed in~\cite{Toft202NewMinorClosedTransversal}. We omit a proof of this to keep the current paper thematically coherent.    

Despite the lemma and remark above, the property of being minor-closed fails for rook matroids: there are rook matroids whose deletion is not a rook matroid. The deeper reason for this failure is best understood in terms of the observation that minors of fundamental transversal matroids are rarely fundamental transversal. We thank Joe Bonin for suggesting an example that did not rely on exhaustive computer search, and for pointing out that the class of rook matroids on $332/1$-avoiding skew shapes is \emph{not} minor-closed, as was claimed in an earlier version of this manuscript. The next example makes reference to lattice path matroids; see Section~\ref{subsection:lattice_path_matroids} for the definition. 

\begin{example}\label{example:rook_matroid_not_minor_closed}
    Consider $\lambda / \mu = 442/2$, the skew shape shown on the left in Figure~\ref{fig:Q6_geometric_rep} and $M = \rookMat_{\lambda / \mu}$. The matroid $M\setminus 2$ can be seen to be isomorphic to the lattice path matroid $\pathMat_{332/1}$, which is not a rook matroid. This is because  the only possible candidate is $\rookMat_{332/1} \cong Q_{6}$ (proven below), which is a known excluded minor of lattice path matroids. Thus $M\setminus 2$ is not a rook matroid and rook matroids are not minor-closed.
\end{example}    
\begin{figure}[!ht]
    \centering
    \begin{subfigure}{0.30\textwidth}
        \centering
        \begin{tikzpicture}[inner sep=0in,outer sep=0in]
\node (n) {\begin{varwidth}{6cm}{
\ytableausetup{boxsize=1.25em}
         \begin{ytableau} \none & \none[4] & \none[5] & \none[6] & \none[7]  \\ \none[1] & \none & \none &  &  \\ \none[2] &  &  &  &  \\ \none[3] &  &  \\ \end{ytableau}
        }\end{varwidth}};
    \end{tikzpicture}
    \end{subfigure}
    \begin{subfigure}{0.30\textwidth}
\centering
   \begin{tikzpicture}[inner sep=0in,outer sep=0in]
\node (n) {\begin{varwidth}{6cm}{
\ytableausetup{boxsize=1.35em}
\begin{ytableau} \none & \none[4] & \none[5] & \none[6]  \\ \none[1] & \none &  &  \\ \none[2] &  &  &  \\ \none[3] &  &  \\ \end{ytableau}
}\end{varwidth}};
\end{tikzpicture}
\end{subfigure}
\hspace{-1 cm}
~
\begin{subfigure}{0.30\textwidth}
\centering
 \begin{tikzpicture}
\node[dot=6pt, fill=black, label={[font=\Large]left:{$b$}}] at (0, 0) (b) {};
\node[dot=6pt, fill=black, label={[font=\Large]above:{$c$}}] at ($(b)+(30:1cm)$) (c) {};
\node[dot=6pt, fill=black, label={[font=\Large]above:{$e$}}] at ($(c)+(30:1cm)$) (e) {};
\node[dot=6pt, fill=black, label={[font=\Large]below:{$d$}}] at ($(b)+(330:1cm)$) (d) {};
\node[dot=6pt, fill=black, label={[font=\Large]below:{$a$}}] at ($(d)+(330:1cm)$) (a) {};

\node[dot=6pt, fill=black, label={[font=\Large]right:{$f$}}] at ($(b)+(4, 0)$) (f) {};
\draw[very thick,black] (b)-- (c) -- (e);
\draw[very thick,black] (b)-- (d) -- (a);
\end{tikzpicture}

\end{subfigure}
 \caption{Left: rook matroid $M = 442/2$ such that $M \setminus 2$ is not a rook matroid. Middle: rook matroid $\rookMat_{332/1}$. Right: the geometric representation of the $Q_{6}$ matroid; $\rookMat_{332/1}$ and $Q_{6}$ are isomorphic.}
 \label{fig:Q6_geometric_rep}
\end{figure}

\begin{lemma}\label{lem:3321_Q6_isomorphism}
The rook matroid $\rookMat_{332/1}$ is isomorphic to the quaternary matroid $Q_{6}$.
\end{lemma}

\begin{proof}
By checking the two three-element circuits of each matroid, one sees that the map from $\{1,2,\ldots , 6\}$ to $\{a,b,c,d,e,f\}$ defined as $1 \mapsto c, 2 \mapsto f, 3 \mapsto d, 4 \mapsto a, 5 \mapsto b, 6 \mapsto e$ is a matroid isomorphism between $\rookMat_{332/1}$ and $Q_{6}$.
\end{proof}

We note that $Q_{6}$ is an excluded minor for the class of lattice path matroids (see \cite[Theorem 3.1]{Bonin2010ExcludedMinors} where $Q_{6}$ appears in the list as $A_{3}$). For further properties of $Q_{6}$, see \cite[pg. 641]{Oxley2011MatroidBook}. We explore the link between $332/1$-avoiding shapes and lattice path matroids in greater depth in the next section. 
   
\subsection{Relation to lattice path matroids}\label{subsection:lattice_path_matroids}

In what follows we collect enumerative and matroidal results linking rook matroids to lattice path matroids, a well-studied family of transversal matroids that also arise from skew shapes. See~\cite{Bonin2003lattice, Bonin2006lattice} for an introduction to lattice path matroids. The goal of this subsection is to identify precisely when a rook matroid and a lattice path matroid on the same skew shape are isomorphic. 

 Let $m, r \geq 0$ and $U, L$ be two lattice paths from $(0, 0)$ to $(m, r)$ such that $U$ never goes below $L$. (All lattice paths in this paper only consist of North and East steps.) Let $\lambda / \mu$ be the skew shape defined by the region enclosed by $U$ and $L$. Lattice paths can be uniquely identified by recording the indices of their North steps. 
 
 Denote by $M[U, L]$ the transversal matroid on $[m+r]$ with presentation given by $\mathcal{E} = (N_{1}, \ldots , N_{r})$ where $N_{j}$ are the North step sets defined by \[N_{j} = \{i: \text{$i$ is the index of the $j^\thsup$ North step of some lattice path $P$ in $\lambda/\mu$}\}.\]  

\begin{definition}
    A \defin{lattice path matroid} (LPM) is a matroid that is isomorphic to the matroid $M[U, L]$ for some lattice paths $U, L$. 
\end{definition}

To emphasize the perspective of the matroid arising from a skew shape we will instead use the notation $\defin{\pathMat_{\lambda / \mu}} \coloneqq M[U, L]$,   where $U, L$ are respectively the upper and lower boundary paths of the diagram defined by the (possibly degenerate) skew shape $\lambda / \mu$. After deleting empty rows and columns of $\lambda/\mu$, the boundary paths $U$ and $L$ share no steps, so the associated lattice path matroid is loopless and coloopless. Note that when $\mu$ is the empty partition, the lattice path matroid $\pathMat_{\lambda}$ is called a \defin{Schubert matroid} or generalized Catalan matroid~\cite{Bonin2006lattice}.

Our first result connecting rook matroids and lattice path matroids is a straightforward bijection that naturally leads to the question of when they are isomorphic. 

For any skew shape $\lambda/\mu$,
it is easy to see that $\pathMat_{\lambda/\mu}$ and $\rookMat_{\lambda/\mu}$
have equal cardinality for the following reason. Define an \defin{inner valley} of a lattice path in $\lambda/\mu$ to be a valley such that neither its East step nor its North step lies on the upper boundary of $\lambda/\mu$. Then, given any lattice path inside $\lambda/\mu$, place a rook in each inner valley of the path. This defines a non-nesting rook placement, and this process is clearly reversible. We record this fact below for use later on in Section~\ref{section:rooks_as_linear_extensions}.

\begin{proposition}\label{lem:setBijectionPathRook}
For any skew shape $\lambda/\mu$, the map $T_{\lambda/\mu}: \mathrm{NN}_{\lambda/\mu} \to \pathMat_{\lambda/\mu}$ defined by fixing the inner valleys of the lattice path precisely at cells containing rooks is a bijection.
\end{proposition}

In particular, when $\lambda$ is the staircase $(n,n-1,\dotsc,1)$, we obtain the bijection that yields the formula in Example~\ref{eg:Narayana_poly}. This connection facilitates an exact count of the number of non-nesting rook placements contained 
in $\lambda / \mu$ via the Lindstr\"{o}m--Gessel--Viennot lemma; see~\cite[Theorem 10.7.1]{Krattenthaler2015} and \cite[Eq. (6)]{StanleyPitman2002}, for example.  

The bijection in Proposition~\ref{lem:setBijectionPathRook} is not a matroid isomorphism in general. The goal for the remainder of this 
subsection is to characterize when
$\rookMat_{\lambda / \mu}$ and $\pathMat_{\lambda / \mu}$ are isomorphic. 
As we noted in Section~\ref{subsection:structural_results}, the obstruction to this isomorphism comes from the skew shape $332/1$ where the lattice path matroid structure and rook matroid structure are not isomorphic. The more surprising fact is that this is the only obstruction for a rook matroid to be isomorphic to a lattice path matroid. 

The next sequence of results develops the tools needed to define this matroid isomorphism using the combinatorics of skew shapes. The first notion involves dividing up the skew shape into rectangles. 

\begin{definition}
    The \defin{rectangular decomposition} of a skew shape $\lambda / \mu$ is the collection of rectangles partitioning the skew shape, each of which is formed by grouping together columns that share the same set of row indices. 
\end{definition}

An example of the rectangular decomposition of a skew shape is shown in Figure~\ref{fig:fatSnakeExample_Ferrers}. We use the rectangular decomposition to give a characterization of \defin{$332/1$-avoiding} skew shapes, that is the class of skew shapes \emph{not} containing $332/1$ as a subshape. The following proposition gives three ways of describing this class.

\begin{proposition}\label{prop:q6av}
Let $\lambda / \mu$ be a connected skew shape. 
Then the following are equivalent:
\begin{enumerate}[label=(\roman*)]
  \item The shape $\lambda / \mu$ is $332/1$-avoiding.
  \item The shape $\lambda / \mu$ can 
  be built from the skew shape consisting of a single cell 
  by using a sequence of the following four operations: (I)
  duplicating the first row, (II) duplicating the last column, (III)
  adding a box to the first row, (IV)
  creating a new first row with a single box.
  These four operations are illustrated on the shape below: 
  \[
  \ytableausetup{boxsize=0.8em}
     \ydiagram{2+2,4,4,2} \text{ gives } 
     \ydiagram{2+2,2+2,4,4,2},\;
     \ydiagram{2+3,5,5,2},\;
     \ydiagram{3+2,4,4,2},\;
     \ydiagram{3+1,2+2,4,4,2} \quad \text{respectively}.
  \]
  \item The rectangular decomposition of $\lambda / \mu$ satisfies the property that for any three consecutive rectangles $R_{1}, R_{2}, R_{3}$, we \emph{do not simultaneously have}
  that 
  \begin{enumerate}
     \item the top of $R_{1}$ 
  is strictly above the bottom of $R_{3}$,
  \item the top of $R_{1}$ is strictly below the top of $R_{2}$,
  \item the bottom of $R_{3}$ is strictly above the bottom of $R_{2}$.
  \end{enumerate}
\end{enumerate}
\end{proposition}
\begin{proof}
Let $\lambda/\mu$ occupy $r$ rows and $c$ columns. 
The equivalences are clear if either $r \leq 2$ or $c \leq 2$, so 
assume for the remainder of this proof that $r, c \geq 3$. 

$(i) \implies (ii)$ 
Let $\lambda/\mu$ be a skew shape not satisfying (ii). 
We must show that the diagram must then contain $332/1$ as a subshape. We may assume that $\lambda/\mu$ is a minimal shape having this property. (Here we use minimal with respect to the size $|\lambda/\mu|$ of the skew shape.)

By minimality, every proper subshape of $\lambda/\mu$ satisfies (ii), which implies that none of the inverses of the four operations in (ii) can be applied to $\lambda/\mu$. Hence $\lambda/\mu$ satisfies the following properties:
\begin{enumerate}
    \item The first two rows of $\lambda/\mu$ have different length. 
    \item The last two columns  of $\lambda/\mu$ have different length. 
    \item The first and second row have their last cells in the same column.
    \item The first row contains more than one box.
\end{enumerate}

Properties (1) and (3) imply that $\lambda/\mu$ has an inner corner at the first row; 
let the corresponding column index of that corner be $k$. 
Let the common column index in which the first two rows have their last cell be $j$. 
The last two columns of the skew shape will then be $j-1$ and $j$. By property (2), columns $j-1$ and $j$ have different length which implies that $\lambda/\mu$ has an outer corner in column $j$; let the corresponding row index of that corner be $s \geq 3$, by the assumption on the number of rows and columns. Then $\lambda/\mu$ contains $332/1$ as a subshape since it has an inner corner at $(1, k)$, an outer corner at $(s, j)$ and columns $k$ to $j-1$ --- which by the assumption $r, c \geq 3$ are at least two in number --- all have the same length. 

\begin{figure}[!ht]
\[
\scalebox{1.4}{
\begin{tikzpicture}[baseline=(current bounding box.center)]
\begin{scope}
\clip  (0,0)--(2,0)--(2,1)--(3,1)--(3,3)--(1,3)--(1,2)--(0,2)--cycle;
\fill[gray,opacity=0.3] (0, 0) rectangle (1, 2);
\fill[gray,opacity=0.3] (2, 1) rectangle (3, 3);
\draw[step=1em,gray](0,0) grid (3,3);
\end{scope}
\draw[black,line width=0.9] 
 (0,0)--(2,0)--(2,1)--(3,1)--(3,3)--(1,3)--(1,2)--(0,2)--cycle;
\end{tikzpicture}
}
\]
\caption{
Rectangular decomposition of $332/1$ shown with shaded rectangles. The skew shape $332/1$ is characterized by its rectangles satisfying properties (a), (b), (c) in point (iii) of Proposition~\ref{prop:q6av}.
}\label{fig:rectangular_decomp_Q6}
\end{figure}
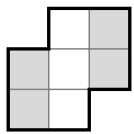

$(ii) \implies (iii)$ Proceed by induction 
on $|\lambda/\mu|$; the base case holds vacuously. 
Assume the result holds for all skew shapes with size less than $|\lambda/\mu|$. 
Let the skew shape $\lambda_{1}/\mu_{1}$ be the shape before $\lambda/\mu$ to which one of the operations (I)-(IV) in point (ii) were applied and 
let $\mathcal{D}$ and $\mathcal{D}'$ respectively be the collection of rectangles in the rectangular decompositions of $\lambda/\mu$ and $\lambda_{1}/\mu_{1}$. 
Consider the operation that transformed $\lambda_{1}/\mu_{1}$ to $\lambda/\mu$. 
If it was (I), then duplicating the top row maintains the relative heights of any three rectangles in $\mathcal{D}'$ and hence (iii) holds for $\mathcal{D}$. 
If it was (II), then only the width of the last rectangle in $\mathcal{D}$ changes while its heights stay the same, so (iii) holds for $\mathcal{D}$. 
If the last operation was (III) and the number of rectangles 
in $\mathcal{D}'$ increases, then the location of the bottom of 
the last rectangle in $\mathcal{D}$ strictly increases, which ensures 
that property (a) and (b) in point (iii) cannot hold simultaneously 
with $R_{3}$ as this last rectangle. 
Finally, if the last operation is (IV), then in the non-trivial case, 
$\mathcal{D}$ has one more rectangle than 
$\mathcal{D}'$ but this rectangle has the same bottom as the last rectangle in $\mathcal{D}'$ so (c) does not hold.

$(iii) \implies (i)$ Suppose $\lambda/\mu$ contains $332/1$. We must show that there are three consecutive rectangles in the rectangular decomposition of $\lambda/\mu$ such that \emph{(a)}, \emph{(b)}, \emph{(c)} hold for these. 

Since $\lambda/\mu$ contains $332/1$ as a subshape, we can then find an inner corner $(i_{1}, j_{1})$ and an outer corner $(i_{2}, j_{2})$ of $\lambda/\mu$ such that the subshape of $\lambda/\mu$ consisting of rows $i_{1}$ to $i_{2}$ and columns $j_{1}+1$ to $j_{2}-1$ is a rectangle; call this $R_{2}$. 
Then if $R_{1}$ is the rectangle consisting of column $j_{1}$ and $R_{3}$ is the rectangle consisting of column $j_{2}$, it is clear that rectangles $R_{1}, R_{2}, R_{3}$ satisfy properties (a), (b), (c). 
See Figure~\ref{fig:rectangular_decomp_Q6} for an example of this on $332/1$ itself. This concludes the proof. 
\end{proof}

We now introduce a tool that will give us a way to interpolate between rook placements and lattice paths. The idea is to represent a lattice path by a permutation the $i^{\thsup}$ entry of which keeps track of which row or column index of the skew shape the $i^{\thsup}$ step of the lattice path occupies, as counted from the bottom. 

\begin{definition}[Path permutation]\label{def:path_permutation}
Let $\lambda/\mu$ be a skew shape with $r$ rows and $c$ columns labeled in the usual way. The \defin{path permutation} corresponding to $L$ is $\psi_{L} = \psi_{1}\ldots \psi_{r+c} \in \mathfrak{S}_{r+c}$, where:
\[
\psi_{i} = \begin{cases}
    \text{column index of the column occupied by step $i$} \quad &\text{if $i$ is an East step,} \\ 
    \text{row index of the row occupied by step $i$} \quad &\text{if $i$ is a North step.}
\end{cases}
\]
\end{definition}

\begin{example}
A lattice path $L$ and its corresponding permutation are shown below:
\[
\begin{tikzpicture}[baseline=(current bounding box.center)]
\begin{scope}
\clip (0,0)--(3,0)--(3,2)--(5,2)--(5,4)--(7,4)--(7,7)--(5,7)--(5,6)--(0,6)--cycle;
\draw[step=1em,gray](0,0) grid (8,8);
\end{scope}
 \draw[black] (0,0)--(3,0)--(3,2)--(5,2)--(5,4)--(7,4)--(7,7)--(5,7)--(5,6)--(0,6)--cycle;

\draw[blue,line width=1.2pt] (0,0)--(1,0)--(1,1)--(1,2)--(1,3)--(2,3)--(2,4)--(3,4)--(3,5)--(4,5)--(4,6)--(7,6)--(7,7);

\draw (-0.5,0.5) node  {$\scriptstyle{7}$};
\draw (-0.5,1.5) node  {$\scriptstyle{6}$};
\draw (-0.5,2.5) node  {$\scriptstyle{5}$};
\draw (-0.5,3.5) node  {$\scriptstyle{4}$};
\draw (-0.5,4.5) node  {$\scriptstyle{3}$};
\draw (-0.5,5.5) node  {$\scriptstyle{2}$};
\draw (-0.5,6.5) node  {$\scriptstyle{1}$};
\draw ( 0.5,7.5) node  {$\scriptstyle{8}$};
\draw ( 1.5,7.5) node  {$\scriptstyle{9}$};
\draw ( 2.5,7.5) node  {$\scriptstyle{10}$};
\draw ( 3.5,7.5) node  {$\scriptstyle{11}$};
\draw ( 4.5,7.5) node  {$\scriptstyle{12}$};
\draw ( 5.5,7.5) node  {$\scriptstyle{13}$};
\draw ( 6.5,7.5) node  {$\scriptstyle{14}$};
\end{tikzpicture}
\qquad
\left[
\begin{array}{cccccccccccccc}
  \scriptstyle{1} & \scriptstyle{2} & \scriptstyle{3} & \scriptstyle{4} & \scriptstyle{5} & \scriptstyle{6} & \scriptstyle{7} & \scriptstyle{8} & \scriptstyle{9} & \scriptstyle{10} & \scriptstyle{11} & \scriptstyle{12} & \scriptstyle{13}  & \scriptstyle{14} \\
  \scriptstyle{8} & \scriptstyle{7} & \scriptstyle{6} & \scriptstyle{5} & \scriptstyle{9} & \scriptstyle{4}  & \scriptstyle{10} & \scriptstyle{3} & \scriptstyle{11} & \scriptstyle{2} & \scriptstyle{12} & \scriptstyle{13} & \scriptstyle{14} & \scriptstyle{1}
 \end{array}
 \right].
\]

\end{example}

We now define a special lattice path inside $\lambda/\mu$ that only exists when $\lambda/\mu$ is $332/1$-avoiding. The permutation corresponding to this path turns out to be precisely the bijection that will induce the matroid isomorphism between $\pathMat_{\lambda/\mu}$ and $\rookMat_{\lambda/\mu}$ shown ahead. In Section~\ref{section:rooks_as_linear_extensions}, the skew shape is given a different labeling; the path permutation with respect to this labeling plays a crucial role in the proof of the main result in that section.    

\begin{definition}[Spine path]\label{definition:spine_path}
Suppose that $\lambda/\mu$ is a $332/1$-avoiding skew shape
contained in an $r \times c$ rectangle.
Let $R_1,\dotsc,R_k$ denote the rectangles in the rectangular decomposition of $\lambda/\mu$. (Note that these rectangles satisfy item (iii) in Proposition~\ref{prop:q6av}.)
We construct a lattice path $L$
of length $r+c$ as follows. 

Start from the lower left hand corner of $R_1$.
\begin{enumerate}
\item[\textrm{(i)}] If there is only one rectangle $(k=1)$ then 
let $L$ be the North--East boundary of $R_{1}$.

 \item[\textrm{(ii)}] 
 If the bottom of $R_2$ is not at the same level as the bottom of $R_1$,
 take North steps until the bottom of $R_2$ is reached,
 and then East steps until
 the lower left corner of $R_2$ is reached.
 Recursively construct the rest of $L$ from here based on the shape $R_2,R_3,\dotsc,R_k$.

\item[\textrm{(iii)}] If $k \geq 3$, and if the bottom of $R_1$ and $R_2$ are the same,
and the top of $R_1$ is at the same level as the bottom of $R_3$,
then take North steps until the top of $R_1$ is reached,
and then take East steps until the bottom left corner of $R_3$
is reached. 
Recursively construct the rest of $L$ based on the shape $R_3,\dotsc,R_k$.

\item[\textrm{(iv)}] In any other case, follow the North--East border of $R_1$.
This ends somewhere on the left side of $R_2$, say at $(p,q)$.
Now construct the rest of $L$ using the part of the shape which is
North--East of $(p,q)$.

Observe that we apply this case if (a) $k=2$,
(b) the bottoms of $R_1$, $R_2$ and $R_3$ are the same,
or (c) when the bottom of $R_3$ is strictly above the top of $R_2$.

(Observe that the case when the bottom of $R_3$ is strictly below the top of $R_1$ does not appear, since this would violate the $332/1$-avoiding property, see Proposition~\ref{prop:q6av}, so we have covered all cases.)
\end{enumerate}
See Figure~\ref{fig:fatSnakeCases} for an illustration of the four cases.

The resulting lattice path is called the \defin{spine path} of $\lambda/\mu$.
By definition, the spine path is always contained inside $\lambda/\mu$.
\end{definition}

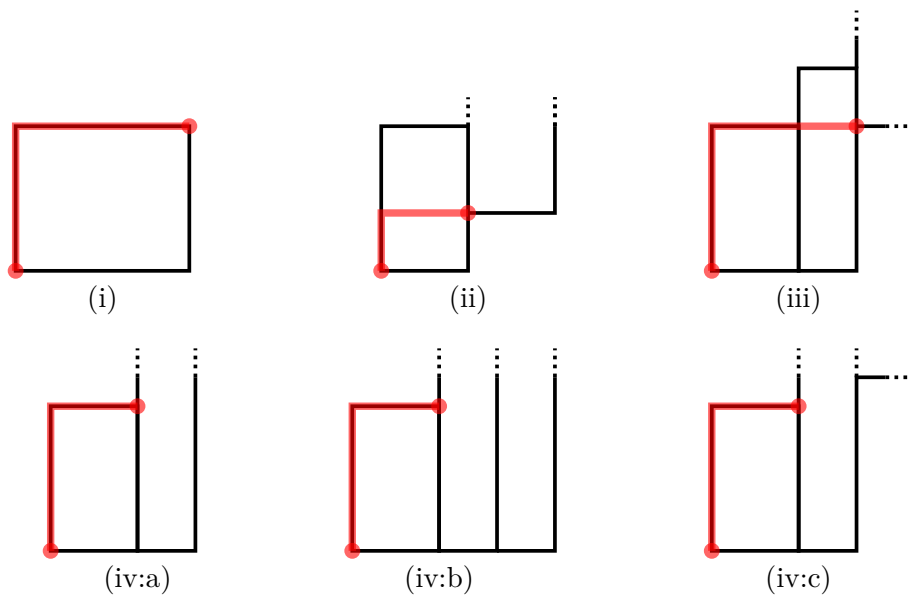
\begin{figure}[!ht]
\begin{align*}
\begin{tikzpicture}
\draw[line width=1.4pt] (0,0)--(6,0)--(6,5)--(0,5)--cycle;
\draw[line width=2.8pt,red,opacity=0.6] (0,0)--(0,5)--(6,5);
\fill[red,opacity=0.6] (0,0) circle (3pt);
\fill[red,opacity=0.6] (6,5) circle (3pt);
\draw (3,-1) node  {(i)};
\end{tikzpicture}
&&
\begin{tikzpicture}
 \draw[line width=1.4pt] (0,0)--(3,0)--(3,5)--(0,5)--cycle;
 \draw[line width=1.4pt] (3,2)--(6,2)--(6,5);
 \draw[line width=1.4pt,dotted] (3,5)--(3,6);
 \draw[line width=1.4pt,dotted] (6,5)--(6,6);
 \draw[line width=1.4pt,dotted] (6,5)--(6,6);
\draw[line width=2.8pt,red,opacity=0.6] (0,0)--(0,2)--(3,2);
\fill[red,opacity=0.6] (0,0) circle (3pt);
\fill[red,opacity=0.6] (3,2) circle (3pt);
 \draw (3,-1) node  {(ii)};
\end{tikzpicture}
&&
\begin{tikzpicture}
 \draw[line width=1.4pt] (0,0)--(3,0)--(3,5)--(0,5)--cycle;
 \draw[line width=1.4pt] (3,0)--(5,0)--(5,7)--(3,7)--cycle;
 \draw[line width=1.4pt] (5,7)--(5,8);
 \draw[line width=1.4pt] (5,5)--(6,5);
 \draw[line width=1.4pt,dotted] (5,8)--(5,9);
 \draw[line width=1.4pt,dotted] (6,5)--(7,5);
 \draw[line width=2.8pt,red,opacity=0.6] (0,0)--(0,5)--(5,5);
 \fill[red,opacity=0.6] (0,0) circle (3pt);
 \fill[red,opacity=0.6] (5,5) circle (3pt);
 \draw (3,-1) node  {(iii)};
\end{tikzpicture}
\\
\begin{tikzpicture}
 \draw[line width=1.4pt] (0,0)--(3,0)--(3,5)--(0,5)--cycle;
 \draw[line width=1.4pt] (3,6)--(3,0)--(5,0)--(5,6);
 \draw[line width=1.4pt,dotted] (3,6)--(3,7);
 \draw[line width=1.4pt,dotted] (5,6)--(5,7);
 \draw[line width=2.8pt,red,opacity=0.6] (0,0)--(0,5)--(3,5);
 \fill[red,opacity=0.6] (0,0) circle (3pt);
 \fill[red,opacity=0.6] (3,5) circle (3pt);
 \draw (3,-1) node  {(iv:a)};
\end{tikzpicture}
&& 
\begin{tikzpicture}
 \draw[line width=1.4pt] (0,0)--(3,0)--(3,5)--(0,5)--cycle;
 \draw[line width=1.4pt] (3,6)--(3,0)--(5,0)--(5,6);
 \draw[line width=1.4pt] (5,0)--(7,0)--(7,6);
 \draw[line width=1.4pt,dotted] (3,6)--(3,7);
 \draw[line width=1.4pt,dotted] (7,6)--(7,7);
 \draw[line width=1.4pt,dotted] (5,6)--(5,7);
 \draw[line width=2.8pt,red,opacity=0.6] (0,0)--(0,5)--(3,5);
 \fill[red,opacity=0.6] (0,0) circle (3pt);
 \fill[red,opacity=0.6] (3,5) circle (3pt);
 \draw (3,-1) node  {(iv:b)};
\end{tikzpicture}
&& 
\begin{tikzpicture}
 \draw[line width=1.4pt] (0,0)--(3,0)--(3,5)--(0,5)--cycle;
 \draw[line width=1.4pt] (3,6)--(3,0)--(5,0)--(5,6)--(6,6);
 \draw[line width=1.4pt,dotted] (3,6)--(3,7);
 \draw[line width=1.4pt,dotted] (6,6)--(7,6);
 \draw[line width=1.4pt,dotted] (5,6)--(5,7);
 \draw[line width=2.8pt,red,opacity=0.6] (0,0)--(0,5)--(3,5);
 \fill[red,opacity=0.6] (0,0) circle (3pt);
 \fill[red,opacity=0.6] (3,5) circle (3pt);
 \draw (3,-1) node  {(iv:c)};
\end{tikzpicture}
\end{align*}
\caption{
The four cases for the construction of the spine path  (shown in thick red) detailed in Definition~\ref{definition:spine_path}.
}\label{fig:fatSnakeCases}
\end{figure}

\begin{example}
The spine path is particularly easy to construct for Ferrers shapes, 
see the first shape in Figure~\ref{fig:fatSnakeExample_Ferrers}.
\begin{figure}[!ht]
\[
\begin{tikzpicture}[baseline=(current bounding box.center)]
\begin{scope}
\clip (0,0)--(2,0)--(2,1)--(5,1)--(5,4)--(6,4)--(6,6)--(7,6)--(7,7)--(0,7)--cycle;
\fill[gray,opacity=0.3] (0, 0) rectangle (2, 7);
\fill[gray,opacity=0.3] (5, 4) rectangle (6, 7);
\draw[step=1em,gray](0,0) grid (7,7);
\end{scope}
\draw[black,line width=0.9] 
(0,0)--(2,0)--(2,1)--(5,1)--(5,4)--(6,4)--(6,6)--(7,6)--(7,7)--(0,7)--cycle;

\draw[red,opacity=0.6,line width=2.8pt] 
(0,0)--(0,1)--(2,1)--(2,4)--(5,4)--(5,6)--(6,6)--(6,7)--(7,7);
\fill[red,opacity=0.6] (0,0) circle (3pt);
\fill[red,opacity=0.6] (7,7) circle (3pt);
\end{tikzpicture}
\qquad 
\begin{tikzpicture}[baseline=(current bounding box.center)]
\begin{scope}
\clip (0,0)--(2,0)--(2,1)--(4,1)--(4,3)--(6,3)--(6,5)--(7,5)--(7,6)--(9,6)--(9,7)--(5,7)--(5,5)--(2,5)--(2,3)--(0,3)--cycle;
\fill[gray,opacity=0.3] (0, 0) rectangle (2, 3);
\fill[gray,opacity=0.3] (4, 3) rectangle (5, 5);
\fill[gray,opacity=0.3] (6, 5) rectangle (7, 7);
\fill[gray,opacity=0.3] (8, 6) rectangle (9, 7);
\draw[step=1em,gray](0,0) grid (9,7);
\end{scope}
\draw[black,line width=0.9] 
 (0,0)--(2,0)--(2,1)--(4,1)--(4,3)--(6,3)--(6,5)--(7,5)--(7,6)--(9,6)--(9,7)--(5,7)--(5,5)--(2,5)--(2,3)--(0,3)--cycle;
\draw[red,opacity=0.6,line width=2.8pt] 
(0,0)--(0,1)--(2,1)--(2,3)--(4,3)--(4,5)--(6,5)--(6,6)--(7,6)--(7,7)--(9,7);
\fill[red,opacity=0.6] (0,0) circle (3pt);
\fill[red,opacity=0.6] (9,7) circle (3pt);
\end{tikzpicture}
\]
\caption{
Two examples of the spine path in two $332/1$-avoiding skew shapes; the rectangular decomposition of the shapes is indicated by the shading.
}\label{fig:fatSnakeExample_Ferrers}
\end{figure}
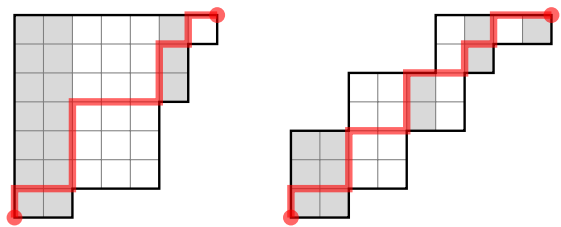
\end{example}

Below, when we refer to a North-East segment of $L$, we mean the word $\mathtt{NE}$ as a subword of $L$ when written as a sequence of $\mathtt{N}$s and $\mathtt{E}$s.

\begin{lemma}\label{lem:spinePathRectangles}
Suppose $L$ is the spine path of a $332/1$-avoiding skew shape $\lambda/\mu$,
and consider a  
corner
of the spine path formed by a North-East segment of $L$ occurring in row $p$ and column $q$ of $\lambda/\mu$. Then the subshape $\{(i, j) \in \lambda/\mu: i \geq p, j \geq q\}$ of $\lambda/\mu$ is a rectangle. 

Similarly, if an East--North corner of $L$ occurs in column $q$ and row $p$, then the subshape $\{(i, j) \in \lambda/\mu: i \leq p, j \leq q\}$ of $\lambda/\mu$ is a rectangle.
\end{lemma}
\begin{remark}
    See Figure~\ref{fig:fatSnakeExample_Ferrers} for an example of this property. The idea behind this lemma is that no run of East steps of the spine path intersects the interior of more than one rectangle in the rectangular decomposition of the skew shape.
\end{remark}

\begin{proof}
We only prove the statement concerning the North-East segments 
of $L$ since the other case can be proved analogously. 
Let $R_{1}, \dotsc, R_{k}$ be the rectangles in the rectangular decomposition of $\lambda/\mu$. We note that each rectangle in the rectangular decomposition of a skew shape can contain at most one subshape induced by a North-East segment; this is because any run of East steps of the spine path intersects the interior of at most one rectangle in the skew shape. Let $D_{1}, \ldots , D_{k}$ be these (possibly empty) Ferrers shapes. If there is indeed a $\mathtt{NE}$ segment of $L$ in $R_{i}$, we will call $D_{i}$ non-empty. 

We proceed by induction on $k$. 
When $k=1$, $\lambda/\mu$ is itself a rectangle and $L$ is, by construction, the North-East boundary of $\lambda/\mu$. Hence the subshape formed by the only North-East segment of $L$ is also a rectangle.

If $k=2$ or $3$, the result follows by a casewise inspection of the relative heights of the rectangles $R_{1}, R_{2}, R_{3}$ and the construction of the spine path in Definition~\ref{definition:spine_path}. 

Suppose now $k > 3$ and the result holds for all skew shapes with at most $k-1$ rectangles in the rectangular decomposition. Let $L$ be the spine path of $\lambda/\mu$. By construction of $L$, the restriction of $L$ to the rectangles $R_{1}, \ldots , R_{k-3}$ is the spine path $L'$ of the skew shape with rectangular decomposition $R_{1}, \ldots , R_{k-3}$. Applying the induction hypothesis to this skew shape and $L'$, we see that each of $D_{1}, \ldots , D_{k-3}$ is a rectangle. 

It remains to consider $D_{i}$ for $i = k-2, k-1$ and $k$. If $D_{k-2}$ is non-empty, one of two cases can occur: (a) $L$ crosses the interior of $R_{k-2}$. This corresponds to case (ii) in the spine path construction and hence $D_{k-2}$ is clearly a rectangle; or else we have (b) $L$ does not cross the interior of $R_{k-2}$, in which instance either case (iii) or (iv) from the spine path construction holds. In each of these cases, the rectangle(s) succeeding $R_{k-2}$ have the same bottom as $R_{k-2}$ so once again $D_{k-2}$ is a rectangle. If $D_{k-1}$ is non-empty, only cases (ii) and (iv a) of the spine path construction can occur; in both cases it is immediate that $D_{k-1}$ is a rectangle. Finally by construction of the spine path, $L$ travels along the North-East boundary of $R_{k}$ and hence $D_{k}$ equals $R_{k}$ itself and is hence a rectangle. 
\end{proof}

For brevity, let \defin{$\psi_{\lambda/\mu}$}
be the path permutation determined by the 
spine path of $\lambda/\mu$ whenever $\lambda/\mu$
is a $332/1$-avoiding skew shape. This is called the \defin{spine bijection}
for short.
We will show that if $\lambda/\mu$ is a $332/1$-avoiding skew shape
then its spine bijection induces 
a bijection from the bases of $\pathMat_{\lambda/\mu}$ to
the bases of $\rookMat_{\lambda/\mu}$. 
By an abuse of notation, we call the induced map between set systems  $\psi_{\lambda/\mu}$.
Note that the spine bijection sends the basis in $\pathMat_{\lambda/\mu}$
corresponding to the spine path 
to the empty non-nesting rook placement in $\rookMat_{\lambda/\mu}$. 
Before proceeding with the proof, we first give an example.

\begin{example}
Consider the case $\lambda=321$.
The spine bijection $\psi_\lambda$ (in 2-line notation) and the corresponding spine path are shown below.
\[
\left[
 \begin{array}{cccccc}
  1 & 2 & 3 & 4 & 5 & 6 \\
  3 & 4 & 2 & 5 & 1 & 6
 \end{array}
 \right]
\qquad \qquad
\begin{tikzpicture}[baseline=(current bounding box.center)]
\begin{scope}
\clip (0,0)--(1,0)--(1,1)--(2,1)--(2,2)--(3,2)--(3,3)--(0,3)--(0,0)--cycle;
\draw[step=1em,gray] (0,0) grid (3,3);
\end{scope}
\draw[gray,line width=1.2]
(0,0)--(1,0)--(1,1)--(2,1)--(2,2)--(3,2)--(3,3)--(0,3)--(0,0);
\draw[red,opacity=0.9,line width=1.9]
(0,0)--(0,1)--(1,1)--(1,2)--(2,2)--(2,3)--(3,3);
\draw (0,1)node[entries]{$\scriptstyle{3}$};
\draw (0,2)node[entries]{$\scriptstyle{2}$};
\draw (0,3)node[entries]{$\scriptstyle{1}$};
\draw (1,4)node[entries]{$\scriptstyle{4}$};
\draw (2,4)node[entries]{$\scriptstyle{5}$};
\draw (3,4)node[entries]{$\scriptstyle{6}$};
\end{tikzpicture}
\]
We can see by inspection that $\psi_\lambda$ induces a bijection
from lattice paths in $\pathMat_{321}$ to
non-nesting rook placements in $\rookMat_{321}$.
We illustrate the induced bijection $\psi_{321}: \pathMat_{321} \to \rookMat_{321}$ below by drawing the
lattice paths together with the
corresponding non-nesting rook placements.
Each shape is labeled with the lattice path basis on the left
and its corresponding non-nesting rook placement basis on the right. 
\[
\includegraphics[page=1,width=0.1\textwidth]{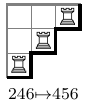}
\quad
\includegraphics[page=2,width=0.1\textwidth]{pics/catalan-spine-bij.pdf}
\quad
\includegraphics[page=3,width=0.1\textwidth]{pics/catalan-spine-bij.pdf}
\quad
\includegraphics[page=4,width=0.1\textwidth]{pics/catalan-spine-bij.pdf}
\quad
\includegraphics[page=5,width=0.1\textwidth]{pics/catalan-spine-bij.pdf}
\quad
\includegraphics[page=6,width=0.1\textwidth]{pics/catalan-spine-bij.pdf}
\quad
\includegraphics[page=7,width=0.1\textwidth]{pics/catalan-spine-bij.pdf}
\]
\[
\includegraphics[page=8,width=0.1\textwidth]{pics/catalan-spine-bij.pdf}
\quad
\includegraphics[page=9,width=0.1\textwidth]{pics/catalan-spine-bij.pdf}
\quad
\includegraphics[page=10,width=0.1\textwidth]{pics/catalan-spine-bij.pdf}
\quad
\includegraphics[page=11,width=0.1\textwidth]{pics/catalan-spine-bij.pdf}
\quad
\includegraphics[page=12,width=0.1\textwidth]{pics/catalan-spine-bij.pdf}
\quad
\includegraphics[page=13,width=0.1\textwidth]{pics/catalan-spine-bij.pdf}
\quad
\includegraphics[page=14,width=0.1\textwidth]{pics/catalan-spine-bij.pdf}
\]
\end{example}

 A crucial ingredient in generalizing the idea above is the presence of a spine path, which we know can only be defined for $332/1$-avoiding shapes. The following theorem shows that a rook matroid and lattice path matroid on the same skew Ferrers shape are isomorphic if and only if $332/1$ does not occur as a subshape of $\lambda /\mu$.

Before we prove this, we introduce a definition that allows one to move between lattice paths contained inside a skew shape.  

\begin{definition}\label{definition:flip_of_lattice_paths}
Two lattice paths $L$ and $L'$ are related by a \defin{flip} if $L$ can be obtained from $L'$ by either (a) flipping an East-North segment of $L'$  to a North-East segment, or (b) flipping a North-East segment to an East-North segment.   
\end{definition}

We are now ready to state the main theorem of this subsection. Since empty rows correspond to coloops and empty columns correspond to loops in both $\rookMat_{\lambda/\mu}$ and $\pathMat_{\lambda/\mu}$, it is enough to prove the theorem below for non-degenerate skew shapes; one extends it by matching these corresponding coloops and loops.

\emph{Note:} A matroidal proof of the first statement in Theorem~\ref{thm:rookPathBij} was later obtained in \cite{BoninDeMier2025EqualConfig} using properties of the configuration of a transversal matroid.

\begin{theorem}\label{thm:rookPathBij}
The lattice path matroid $\pathMat_{\lambda/\mu}$
is isomorphic to the rook matroid $\rookMat_{\lambda/\mu}$ if and only if $\lambda /\mu$ is a $332/1$-avoiding skew shape. In that case, an explicit isomorphism from $\pathMat_{\lambda/\mu}$ to $\rookMat_{\lambda/\mu}$ is given by $\psi_{\lambda/\mu}$.
\end{theorem}
\begin{proof}
For the only if direction: suppose $\lambda /\mu$ contains $332/1$ as a subshape. Then $\rookMat_{\lambda/\mu}$ would contain $Q_{6}$ as a minor by Lemma~\ref{lem:3321_Q6_isomorphism} and Lemma~\ref{lem:easyRookMinors}. 
But $Q_{6}$ is an excluded minor (see \cite[Thm. 3.1]{Bonin2010ExcludedMinors}) 
for the class of lattice path matroids so $\rookMat_{\lambda /\mu}$ 
cannot be isomorphic to $\pathMat_{\lambda /\mu}$. 

For the if direction: let $\lambda /\mu$ be a $332/1$-avoiding skew shape on $r$ rows 
and $c$ columns. Since $\lambda/\mu$ avoids $332/1$, we can consider 
its spine path $L_{0}$ and the associated spine bijection $\psi_{\lambda/\mu}$, 
which for the purpose of brevity we 
occasionally refer to as $\psi$. 
Recall that we are identifying a lattice path by its set of North steps. We will be done by the following claim: 

\textbf{Claim:} If $L$ is a lattice path contained inside $\lambda /\mu$, then $\psi_{\lambda/\mu}(L)$ 
defines a non-nesting rook placement on $\lambda /\mu$. 

\textbf{Proof of claim:} We proceed by induction on $k$, the number of cells between $L$ and the spine path. For the $k = 0$ case, $L$ is the spine path itself, the North steps of which, by construction, correspond to the row indices of $\lambda /\mu$ which implies that $\psi_{\lambda /\mu}(L_{0})$ is the empty rook placement.

In the two-line notation of the spine path permutation, we partition the columns into \defin{blocks},
formed by runs of steps in the same direction (North or East).

Now suppose $k \geq 1$ and consider any lattice path $L^{+}$ contained inside $\lambda/ \mu$ with $k$ cells between it and $L_{0}$. Then $L^{+}$ can be obtained from some lattice path $L$ in $\lambda /\mu$ with $k-1$ cells between $L$ and $L_{0}$ via a flip between steps $i$ and $i+1$. Here $i$ and $i+1$ are consecutive steps of the lattice path $L$ such that exactly one of them is a North step. This means that $L^{+} = (L \setminus \{i\} ) \cup \{i+1\}$ or $L^{+} = (L \setminus \{i + 1\}) \cup \{i\}$. 

In the two-line notation of $\psi_{\lambda/\mu}$, circle the North steps of $L$, as in Example~\ref{eg:spine_bijection_eg}. By the induction hypothesis, the values below the circled entries, $\psi(L)$, define a valid non-nesting rook placement $\rho_{L}$ on $\lambda/\mu$. The last sentence of the paragraph above implies that the set $\psi(L^{+})$ is obtained from $\psi(L)$ by swapping out one of $\psi(i)$ or $\psi(i+1)$ for the other. At the level of the circled entries of the first line of $\psi_{\lambda /\mu}$, the flipping of a $\mathtt{NE}$ segment to an $\mathtt{EN}$ segment or vice versa corresponds to \emph{moving a circle} (the index of a North step) either one position left or right.

We will show that the swapping of $\psi(i)$ and $\psi(i+1)$ corresponds to a movement of some rooks of $\rho_{L}$ that still defines a valid non-nesting rook placement on $\lambda /\mu$. Recall that we have partitioned the columns of the two-line array of the spine bijection $\psi_{\lambda/\mu}$ into blocks. There are four cases to consider:
\begin{itemize}
 \item[A.]  \emph{$\psi(i)$ and $\psi(i+1)$ are both column indices from the same block.}
 \item[B.] \emph{$\psi(i)$ and $\psi(i+1)$ are both row indices from the same block.}
\item[C.] \emph{$\psi(i)$ and $\psi(i+1)$ are in different blocks,
and the North step (exactly one of $i$ or $i+1$) is mapped to a column index.}
\item[D.] \emph{$\psi(i)$ and $\psi(i+1)$ are in different blocks,
and the North step of $L$ is mapped to a row index.}
\end{itemize}

By Lemma~\ref{lem:spinePathRectangles},  any such flip occurs in a rectangular 
region determined by the spine path.
Note that such a rectangular region 
is given by row and column indices in two adjacent blocks
in the two-line notation of $\psi_{\lambda/\mu}$.

\medskip We will see that in every case this corresponds to moving the rooks of $\rho_{L}$ such that they still stay on $\lambda/\mu$, 
since all operations are performed on some \emph{rectangle}, as guaranteed by Lemma~\ref{lem:spinePathRectangles}. We consider each of the cases above.

\begin{itemize}
 \item[A.] In this case we have two adjacent columns $\psi(i)$ and $\psi(i+1)$ where one is empty
 the other is occupied by a rook of $\rho_{L}$. 
 We can then move the rook horizontally to the adjacent 
 column and obtain a new valid rook placement on $\lambda/\mu$. 

\item[B.] In this case we have two adjacent rows $\psi(i)$ and $\psi(i+1)$ where one is empty
 the other is occupied by a rook. The rook is moved vertically to the adjacent row to obtain a new valid rook placement on $\lambda/\mu$.

\item[C.] In this case, either $p = \psi(i)$ and $q = \psi(i+1)$ are respectively an occupied row and an occupied column, or vice versa.

In the first case, we need to make both row $p$ and column $q$ unoccupied without affecting the occupancy of the other rows or columns, 
or the non-nesting nature of the other rooks of $\rho_{L}$. (In the two-line array, this corresponds to moving a circle one step to the left.). 

There are two configurations to consider. 
Either the cell $(p, q)$ is occupied by a rook or not. 
In the former configuration, simply delete the rook in $(p, q)$. This is shown in the figure on the left below.
\begin{center}
\includegraphics[page=6,scale=0.6]{pics/mainBij.pdf}
\raisebox{5.0\height}{$\;\to$}
\includegraphics[page=5,scale=0.6]{pics/mainBij.pdf}
\hspace{2cm}
\includegraphics[page=8,scale=0.6]{pics/mainBij.pdf}
\raisebox{5.0\height}{$\;\to$}
\includegraphics[page=7,scale=0.6]{pics/mainBij.pdf}
\end{center}
In the second configuration, we form a ``staircase'' of cells (marked $\circ$)
based on the location of the existing rooks in the rectangle North-West to $(p, q)$,
as seen in the figure above on the right.
We then move the rooks of $\rho_{L}$ to the locations marked $\circ$, once again obtaining a new valid non-nesting rook placement on $\lambda/\mu$, without affecting the occupancy of any other row or column of $\rho_{L}$.

In the second case, $p$ is a column index and $q$ is a row index. The same operations as above are performed with the figures rotated by $180^\circ$. (In the two-line array, this corresponds to moving a circle one step to the right.)

\item[D.] As in C, there are two cases. In the first case, 
column $q = \psi(i)$ and row $p = \psi(i+1)$ are both unoccupied and need to be made occupied; the second case, where $q$ is a row and $p$ is a column, is treated analogously. 
Again, two configurations may arise: one where the rectangle strictly South-West of $(p, q)$ contains a rook, 
and one where it does not. In the former, simply place a rook at $(p, q)$. In the latter, perform the simultaneous shifting of rooks as in Case C. This is illustrated below.

\begin{center}
\includegraphics[page=5,scale=0.6,angle=180,origin=c]{pics/mainBij.pdf}
\raisebox{5.0\height}{$\;\to$}
\includegraphics[page=6,scale=0.6,angle=180,origin=c]{pics/mainBij.pdf}
\hspace{2cm}
\includegraphics[page=7,scale=0.6,angle=180,origin=c]{pics/mainBij.pdf}
\raisebox{5.0\height}{$\;\to$}
\includegraphics[page=8,scale=0.6,angle=180,origin=c]{pics/mainBij.pdf}
\end{center}
\end{itemize}

Once again, the second configuration in this case corresponds to the same picture rotated $180^{\circ}$. The consideration of the four cases is thus complete.

This shows that $\psi_{\lambda/\mu}(L^{+})$ also defines a non-nesting rook placement on $\lambda/\mu$, which proves the claim. Thus $\psi_{\lambda/\mu}$ maps bases of $\pathMat_{\lambda/\mu}$ to bases of $\rookMat_{\lambda/\mu}$. By Proposition~\ref{lem:setBijectionPathRook}, $\mathcal{B}(\pathMat_{\lambda/\mu})$ and $\mathcal{B}(\rookMat_{\lambda/\mu})$ have equal cardinalities, and hence $\psi_{\lambda/\mu}$ is a matroid isomorphism, completing the proof.
\end{proof}

\begin{example}\label{eg:spine_bijection_eg}
We illustrate the steps of the proof of the if direction of the previous theorem on the $332/1$-avoiding skew shape $7775533/5$. The two-line notation of the spine bijection $\psi_{7775533/5}$ with indicated blocks is:
\[
\psi_{7775533/5} =
\left[
 \begin{array}{cc|ccc|cc|cc|ccc|cc}
  \scriptstyle{1} &  \scriptstyle{2} & \scriptstyle{3} & \scriptstyle{4} & \scriptstyle{5} & \scriptstyle{6} & \scriptstyle{7} & \scriptstyle{8} & \scriptstyle{9} & \scriptstyle{10} & \scriptstyle{11} & \scriptstyle{12}& \scriptstyle{13} & \scriptstyle{14} \\
  \scriptstyle{7} & \scriptstyle{6} & \scriptstyle{8} & \scriptstyle{9} & \scriptstyle{10} & \scriptstyle{5} & \scriptstyle{4} & \scriptstyle{11} & \scriptstyle{12} & \scriptstyle{3} & \scriptstyle{2} &
  \scriptstyle{1} & \scriptstyle{13} & \scriptstyle{14} 
 \end{array}\right].
\]

See Figure~\ref{fig:isomorphism_thm_eg} below for an illustration of the spine path in red, a lattice path $L$ in blue, and the four steps at which flips can be performed to produce $L^{+}$. These four cases appear in the blue lattice path $L$ shown in Figure~\ref{fig:isomorphism_thm_eg}.
We have a Case A flip for step $i=4$, Case B for step $i=1$, 
Case C for step $i=9$ and Case D for step $i=7$. 

\begin{figure}[!ht]
    \centering
\begin{tikzpicture}[baseline=(current bounding box.center)]
\begin{scope}
\clip (0,0)--(3,0)--(3,2)--(5,2)--(5,4)--(7,4)--(7,7)--(5,7)--(5,6)--(0,6)--cycle;
\draw[step=1em,gray](0,0) grid (8,8);
\end{scope}
 \draw[black,line width=1.4] (0,0)--(3,0)--(3,2)--(5,2)--(5,4)--(7,4)--(7,7)--(5,7)--(5,6)--(0,6)--cycle;

\draw[red,opacity=0.6,line width=1.8pt] (0,0)--(0,2)--(3,2)--(3,4)--(5,4)--(5,7)--(7,7);
\fill[red,opacity=0.6] (0,0) circle (3pt);
\fill[red,opacity=0.6] (7,7) circle (3pt);

\draw[blue,line width=1.2pt] (0,0)--(0,1)--(1,1)--(1,2)--(2,2)--(2,3)--(2,4)--(3,4)--(3,5)--(4,5)--(4,6)--(7,6)--(7,7);

\draw (0.5,1.5) node  {$\rook$};
\draw (6.5,6.5) node  {$\rook$};
\draw (3.5,5.5) node  {$\rook$};
\draw (2.5,3.5) node  {$\rook$};
\fill[gray,opacity=0.3] (1, 2) rectangle (2, 3);
\fill[gray,opacity=0.3] (0, 0) rectangle (1, 1);
\fill[gray,opacity=0.3] (2, 4) rectangle (3, 5);
\fill[gray,opacity=0.3] (3, 5) rectangle (4, 6);

\draw (-0.5,0.5) node  {$\scriptstyle{7}$};
\draw (-0.5,1.5) node  {$\scriptstyle{6}$};
\draw (-0.5,2.5) node  {$\scriptstyle{5}$};
\draw (-0.5,3.5) node  {$\scriptstyle{4}$};
\draw (-0.5,4.5) node  {$\scriptstyle{3}$};
\draw (-0.5,5.5) node  {$\scriptstyle{2}$};
\draw (-0.5,6.5) node  {$\scriptstyle{1}$};
\draw ( 0.5,7.5) node  {$\scriptstyle{8}$};
\draw ( 1.5,7.5) node  {$\scriptstyle{9}$};
\draw ( 2.5,7.5) node  {$\scriptstyle{10}$};
\draw ( 3.5,7.5) node  {$\scriptstyle{11}$};
\draw ( 4.5,7.5) node  {$\scriptstyle{12}$};
\draw ( 5.5,7.5) node  {$\scriptstyle{13}$};
\draw ( 6.5,7.5) node  {$\scriptstyle{14}$};
\end{tikzpicture}
 \caption{The skew shape $\lambda/\mu = 7775533/5$ together with its spine path $L_{0}$ in red, a representative lattice path $L = \{1,3,5,6,8,10,14\}$ in blue and the corresponding non-nesting rook placement $\psi_{\lambda /\mu}(L) = \{3,5,7,8,10,11,14\}$. From bottom to top, the four shaded cells correspond to $\mathtt{NE}$ or $\mathtt{EN}$ segments the flips of which are of type B, A, C, and D respectively.}
    \label{fig:isomorphism_thm_eg}
\end{figure}
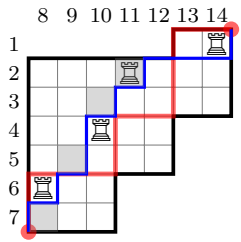

The circled entries of $L$ in $\psi_{\lambda/\mu}$ are shown below:

\[
\psi_{7775533/5} =
\left[
 \begin{array}{cc|ccc|cc|cc|ccc|cc}
  \circled{$\scriptstyle{1}$} &  \scriptstyle{2} & \circled{$\scriptstyle{3}$} & \scriptstyle{4} & \circled{$\scriptstyle{5}$} & \circled{$\scriptstyle{6}$} & \scriptstyle{7} & \circled{$\scriptstyle{8}$} & \scriptstyle{9} & \circled{$\scriptstyle{10}$} & \scriptstyle{11} & \scriptstyle{12} & \scriptstyle{13} & \circled{$\scriptstyle{14}$} \\
  \scriptstyle{7} & \scriptstyle{6} & \scriptstyle{8} & \scriptstyle{9} & \scriptstyle{10} & \scriptstyle{5} & \scriptstyle{4} & \scriptstyle{11} & \scriptstyle{12} & \scriptstyle{3} & \scriptstyle{2} &
  \scriptstyle{1} & \scriptstyle{13} & \scriptstyle{14} 
 \end{array}\right].
\]

We describe the results of performing the four types of flips in Figure~\ref{fig:isomorphism_thm_eg} on $L = \{1,3,5,6,8,10,14\}$.

\begin{itemize}
 \item The flip of type A at step $4$ gives the new path
 $L^+=\{1,3,4,6,8,10,14\}$.
 \item The flip of type B at step $1$ gives the new path
 $L^+=\{2,3,5,6,8,10,14\}$.
  \item The flip of type C at step $9$ gives the new path
 $L^+=\{1,3,5,6,8,9,14\}$.
  \item The flip of type D at step $7$ gives the new path
 $L^+=\{1,3,5,6,7,10,14\}$.
\end{itemize}

\begin{figure}[!ht]
\centering
\begin{tikzpicture}[baseline=(current bounding box.center)]
\begin{scope}
\clip (0,0)--(3,0)--(3,2)--(5,2)--(5,4)--(7,4)--(7,7)--(5,7)--(5,6)--(0,6)--cycle;
\draw[step=1em,gray](0,0) grid (8,8);
\end{scope}
 \draw[black,line width=1.4] (0,0)--(3,0)--(3,2)--(5,2)--(5,4)--(7,4)--(7,7)--(5,7)--(5,6)--(0,6)--cycle;

\draw[red,opacity=0.6,line width=1.8pt] (0,0)--(0,2)--(3,2)--(3,4)--(5,4)--(5,7)--(7,7);
\fill[red,opacity=0.6] (0,0) circle (3pt);
\fill[red,opacity=0.6] (7,7) circle (3pt);
\fill[gray,opacity=0.3] (1, 2) rectangle (2, 3);

\draw[blue,line width=1.2pt] (0,0)--(0,1)--(1,1)--(1,2)--(1,3)--(2,3)--(2,4)--(3,4)--(3,5)--(4,5)--(4,6)--(7,6)--(7,7);

\draw (0.5,1.5) node  {$\rook$};
\draw (6.5,6.5) node  {$\rook$};
\draw (3.5,5.5) node  {$\rook$};
\draw (1.5,3.5) node  {$\rook$};

\draw (-0.5,0.5) node  {$\scriptstyle{7}$};
\draw (-0.5,1.5) node  {$\scriptstyle{6}$};
\draw (-0.5,2.5) node  {$\scriptstyle{5}$};
\draw (-0.5,3.5) node  {$\scriptstyle{4}$};
\draw (-0.5,4.5) node  {$\scriptstyle{3}$};
\draw (-0.5,5.5) node  {$\scriptstyle{2}$};
\draw (-0.5,6.5) node  {$\scriptstyle{1}$};
\draw ( 0.5,7.5) node  {$\scriptstyle{8}$};
\draw ( 1.5,7.5) node  {$\scriptstyle{9}$};
\draw ( 2.5,7.5) node  {$\scriptstyle{10}$};
\draw ( 3.5,7.5) node  {$\scriptstyle{11}$};
\draw ( 4.5,7.5) node  {$\scriptstyle{12}$};
\draw ( 5.5,7.5) node  {$\scriptstyle{13}$};
\draw ( 6.5,7.5) node  {$\scriptstyle{14}$};
\end{tikzpicture}
\hspace{0.5cm}
\begin{tikzpicture}[baseline=(current bounding box.center)]
\begin{scope}
\clip (0,0)--(3,0)--(3,2)--(5,2)--(5,4)--(7,4)--(7,7)--(5,7)--(5,6)--(0,6)--cycle;
\draw[step=1em,gray](0,0) grid (8,8);
\end{scope}
 \draw[black,line width=1.4] (0,0)--(3,0)--(3,2)--(5,2)--(5,4)--(7,4)--(7,7)--(5,7)--(5,6)--(0,6)--cycle;

\draw[red,opacity=0.6,line width=1.8pt] (0,0)--(0,2)--(3,2)--(3,4)--(5,4)--(5,7)--(7,7);
\fill[red,opacity=0.6] (0,0) circle (3pt);
\fill[red,opacity=0.6] (7,7) circle (3pt);

\draw[blue,line width=1.2pt] (0,0)--(1,0)--(1,1)--(1,2)--(2,2)--(2,3)--(2,4)--(3,4)--(3,5)--(4,5)--(4,6)--(7,6)--(7,7);

\fill[gray,opacity=0.3] (0, 0) rectangle (1, 1);

\draw (0.5,0.5) node  {$\rook$};
\draw (6.5,6.5) node  {$\rook$};
\draw (3.5,5.5) node  {$\rook$};
\draw (2.5,3.5) node  {$\rook$};

\draw (-0.5,0.5) node  {$\scriptstyle{7}$};
\draw (-0.5,1.5) node  {$\scriptstyle{6}$};
\draw (-0.5,2.5) node  {$\scriptstyle{5}$};
\draw (-0.5,3.5) node  {$\scriptstyle{4}$};
\draw (-0.5,4.5) node  {$\scriptstyle{3}$};
\draw (-0.5,5.5) node  {$\scriptstyle{2}$};
\draw (-0.5,6.5) node  {$\scriptstyle{1}$};
\draw ( 0.5,7.5) node  {$\scriptstyle{8}$};
\draw ( 1.5,7.5) node  {$\scriptstyle{9}$};
\draw ( 2.5,7.5) node  {$\scriptstyle{10}$};
\draw ( 3.5,7.5) node  {$\scriptstyle{11}$};
\draw ( 4.5,7.5) node  {$\scriptstyle{12}$};
\draw ( 5.5,7.5) node  {$\scriptstyle{13}$};
\draw ( 6.5,7.5) node  {$\scriptstyle{14}$};
\end{tikzpicture}
\hspace{0.5cm}
\begin{tikzpicture}[baseline=(current bounding box.center)]
\begin{scope}
\clip (0,0)--(3,0)--(3,2)--(5,2)--(5,4)--(7,4)--(7,7)--(5,7)--(5,6)--(0,6)--cycle;
\draw[step=1em,gray](0,0) grid (8,8);
\end{scope}
 \draw[black,line width=1.4] (0,0)--(3,0)--(3,2)--(5,2)--(5,4)--(7,4)--(7,7)--(5,7)--(5,6)--(0,6)--cycle;

\draw[red,opacity=0.6,line width=1.8pt] (0,0)--(0,2)--(3,2)--(3,4)--(5,4)--(5,7)--(7,7);
\fill[red,opacity=0.6] (0,0) circle (3pt);
\fill[red,opacity=0.6] (7,7) circle (3pt);

\draw[blue,line width=1.2pt] (0,0)--(0,1)--(1,1)--(1,2)--(2,2)--(2,3)--(2,4)--(3,4)--(3,5)--(3,6)--(4,6)--(7,6)--(7,7);

\draw (0.5,1.5) node  {$\rook$};
\draw (6.5,6.5) node  {$\rook$};
\draw (3.5,4.5) node  {$\rook$};
\draw (4.5,5.5) node  {$\rook$};
\draw (2.5,3.5) node  {$\rook$};

\fill[gray,opacity=0.3] (3, 5) rectangle (4, 6);

\draw (-0.5,0.5) node  {$\scriptstyle{7}$};
\draw (-0.5,1.5) node  {$\scriptstyle{6}$};
\draw (-0.5,2.5) node  {$\scriptstyle{5}$};
\draw (-0.5,3.5) node  {$\scriptstyle{4}$};
\draw (-0.5,4.5) node  {$\scriptstyle{3}$};
\draw (-0.5,5.5) node  {$\scriptstyle{2}$};
\draw (-0.5,6.5) node  {$\scriptstyle{1}$};
\draw ( 0.5,7.5) node  {$\scriptstyle{8}$};
\draw ( 1.5,7.5) node  {$\scriptstyle{9}$};
\draw ( 2.5,7.5) node  {$\scriptstyle{10}$};
\draw ( 3.5,7.5) node  {$\scriptstyle{11}$};
\draw ( 4.5,7.5) node  {$\scriptstyle{12}$};
\draw ( 5.5,7.5) node  {$\scriptstyle{13}$};
\draw ( 6.5,7.5) node  {$\scriptstyle{14}$};
\end{tikzpicture}
\hspace{0.5cm}
\begin{tikzpicture}[baseline=(current bounding box.center)]
\begin{scope}
\clip (0,0)--(3,0)--(3,2)--(5,2)--(5,4)--(7,4)--(7,7)--(5,7)--(5,6)--(0,6)--cycle;
\draw[step=1em,gray](0,0) grid (8,8);
\end{scope}
 \draw[black,line width=1.4] (0,0)--(3,0)--(3,2)--(5,2)--(5,4)--(7,4)--(7,7)--(5,7)--(5,6)--(0,6)--cycle;

\draw[red,opacity=0.6,line width=1.8pt] (0,0)--(0,2)--(3,2)--(3,4)--(5,4)--(5,7)--(7,7);
\fill[red,opacity=0.6] (0,0) circle (3pt);
\fill[red,opacity=0.6] (7,7) circle (3pt);

\draw[blue,line width=1.2pt] (0,0)--(0,1)--(1,1)--(1,2)--(2,2)--(2,3)--(2,4)--(2,5)--(3,5)--(4,5)--(4,6)--(7,6)--(7,7);

\fill[gray,opacity=0.3] (2, 4) rectangle (3, 5);

\draw (0.5,1.5) node  {$\rook$};
\draw (6.5,6.5) node  {$\rook$};
\draw (2.5,5.5) node  {$\rook$};

\draw (-0.5,0.5) node  {$\scriptstyle{7}$};
\draw (-0.5,1.5) node  {$\scriptstyle{6}$};
\draw (-0.5,2.5) node  {$\scriptstyle{5}$};
\draw (-0.5,3.5) node  {$\scriptstyle{4}$};
\draw (-0.5,4.5) node  {$\scriptstyle{3}$};
\draw (-0.5,5.5) node  {$\scriptstyle{2}$};
\draw (-0.5,6.5) node  {$\scriptstyle{1}$};
\draw ( 0.5,7.5) node  {$\scriptstyle{8}$};
\draw ( 1.5,7.5) node  {$\scriptstyle{9}$};
\draw ( 2.5,7.5) node  {$\scriptstyle{10}$};
\draw ( 3.5,7.5) node  {$\scriptstyle{11}$};
\draw ( 4.5,7.5) node  {$\scriptstyle{12}$};
\draw ( 5.5,7.5) node  {$\scriptstyle{13}$};
\draw ( 6.5,7.5) node  {$\scriptstyle{14}$};
\end{tikzpicture}
 \caption{From left to right: the four new rook placements induced by the lattice path $L^{+}$ (shown in blue) after the four types of flips (A, B, D, C). The shaded cell in each skew shape shows where the flip occurred.}
    \label{fig:isomorphism_thm_examples}
\end{figure}
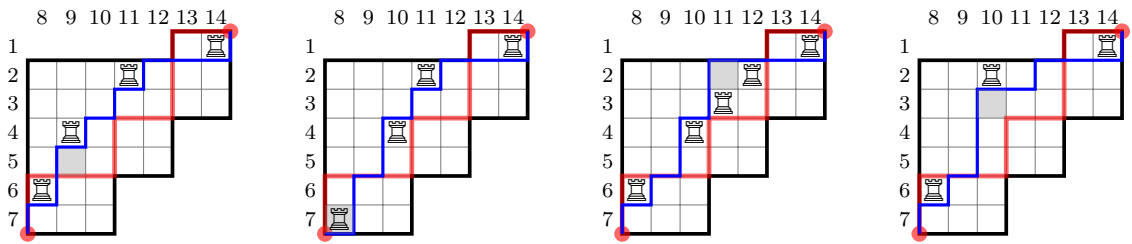

\end{example}

In a different direction, we can obtain lattice 
path matroids as deletions of large enough rook matroids.

\begin{proposition}\label{prop:lpm_as_rook_matroid_minor}
Let $\lambda/\mu$ be a skew shape with $r$ rows and $n-r$ columns. Define partitions $\lambda^{*}, \mu^{*}$ by $\lambda_{i}^{*} = \lambda_{i}+r-i+1$ and $\mu_{i}^{*} = \mu_{i} + r- i$ for $i=1, \ldots , r$. Then $\pathMat_{\lambda/\mu} \cong \rookMat_{\lambda^\ast/\mu^\ast} \setminus [r]$.
\end{proposition}
\begin{proof}
Recall that the sets in the transversal presentation of $\rookMat_{\lambda^{*}/\mu^{*}}$ are \[A_{i} = \{i\} \cup \{r+j: \mu_{i}+r-i < j \leq \lambda_{i}+r-i+1\} \quad  \text{for $i=1, \ldots , r$}.\] By deleting the elements of $[r]$ from $\rookMat_{\lambda^{*}/\mu^{*}}$, we are left with a transversal matroid on $[r+1, r+n]$ whose $i^{\thsup}$ set is the interval equal to the second set in the equation above. By applying the shift $j \mapsto j - r$, and reindexing with $s = r-i+1$, this interval becomes $[\mu_{r-s+1}+s, \lambda_{r-s+1}+s]$ which is the $s^{\thsup}$ set in the transversal presentation of the lattice path matroid $\pathMat_{\lambda/\mu}$. An example is shown in Figure~\ref{fig:LPM_as_rook_minors}. 
\end{proof}

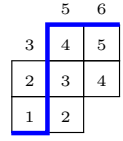
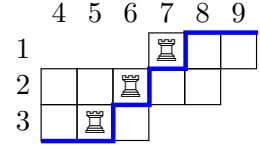
\begin{figure}[!ht]
\centering
\begin{subfigure}[b]{0.40\textwidth}
\centering
   \begin{tikzpicture}[inner sep=0in,outer sep=0in]
\node (n) {\begin{varwidth}{6cm}{
\ytableausetup{textmode, boxsize=1.2em}
\begin{ytableau} 
 \none & \none[ \tiny 5] & \none[ \tiny 6]\\
\none[\tiny 3] & \tiny 4 & \tiny 5  \\ \tiny 2 & \tiny 3 & \tiny 4  \\  \tiny 1 & \tiny 2  \\ \end{ytableau}}\end{varwidth}};
\coordinate (a2) at ($(n.south west)$);
\coordinate (a3) at ($(a2)+(1.25em, 0)$);
\coordinate (a4) at ($(a3)+(0, 1.25em)$);
\coordinate (a5) at ($(a4)+(0, 1.25em)$);
\coordinate (a6) at ($(a5)+(0, 1.25em)$);
\coordinate (a7) at ($(a6)+(1.25em, 0)$);
\coordinate (a8) at ($(a7)+(1.25em, 0)$);
\draw[line width = 0.6mm,blue](a2)--(a3)-- (a4)-- (a5)-- (a6)-- (a7)-- (a8);
\end{tikzpicture}
\caption{Skew shape $\lambda /\mu = 332/1$ with the path $L = 156$ marked in blue.}
\label{fig:LPM_as_rook_minor2} 
\end{subfigure}
\hspace{1 cm}
~
\begin{subfigure}[b]{0.40\textwidth}
\centering
\begin{tikzpicture}[inner sep=0in,outer sep=0in]
\node (n) {\begin{varwidth}{5cm}{
\ytableausetup
{boxsize=1.2em}
  \begin{ytableau} 
  \none & \none[4] & \none[5] & \none[6] & \none[7] & \none[8] & \none[9]  \\ 
  \none[1] & \none & \none & \none & \rook &  &  \\ 
  \none[2] &  &  & \rook &  &  \\ \none[3] &  & \rook  &  \\
\end{ytableau}}\end{varwidth}};
\coordinate (a1) at ([xshift=1.25em]n.south west);
%\node(A1)[] at (a1) {$\bullet$};
\coordinate (a2) at ($([xshift=1.25em]n.south west)+(1.25em, 0)$);
\coordinate (a3) at ($(a2)+(1.25em, 0)$);
\coordinate (a4) at ($(a3)+(0, 1.25em)$);
\coordinate (a5) at ($(a4)+(1.25em, 0)$);
\coordinate (a6) at ($(a5)+(0, 1.25em)$);
\coordinate (a7) at ($(a6)+(1.25em, 0)$);
\coordinate (a8) at ($(a7)+(0, 1.25em)$);
\coordinate (a9) at ($(a8)+(1.25em, 0)$);
\coordinate (a10) at ($(a9)+(1.25em, 0)$);
\draw[line width = 0.6mm,blue] (a1)-- (a2)--(a3)-- (a4)-- (a5)-- (a6)-- (a7)-- (a8)-- (a9)-- (a10);
\end{tikzpicture}
\caption{Skew shape $\lambda^{*} /\mu^{*} = 653/3$ with the path $L'$ marked in blue.}
\label{fig:LPM_as_rook_minor1}
\end{subfigure}
\caption{The map $i \mapsto i+r$ in Proposition~\ref{prop:lpm_as_rook_matroid_minor} adds an $\mathtt{E}$ step before every $\mathtt{N}$ step in $L$ to yield a lattice path $L'$ with a rook in each of its outer corners.}
\label{fig:LPM_as_rook_minors}
\end{figure}

 Rook matroids are shown to be positroids in Section~\ref{subsection:positroids}, and positroids are known to be closed under taking minors~\cite[Proposition 3.5]{ArdilaRinconWilliams2016} and under a cyclic shift of the ground set~\cite[Lemma 3.3]{ArdilaRinconWilliams2016}. Since the map in the proof of Proposition~\ref{prop:lpm_as_rook_matroid_minor} is a cyclic shift, we can thus deduce that every lattice path matroid is a positroid, a fact first proven in \cite{Oh2011Positroids}.

\subsection{Relation to positroids}\label{subsection:positroids}

\begin{definition}
    Let $A$ be a full rank $d \times n$ matrix with real entries such that all its maximal minors are non-negative. The representable matroid $M(A)$ associated to $A$ is called a \defin{positroid}.
\end{definition}

The property of being a positroid is strongly dependent on the total ordering on the ground set of the matroid; the ordering we use here is the standard ordering for the rook matroid introduced in Section~\ref{section:rook_matroid}. The proof method we use is reminiscent of Oh's proof of the positroid property of lattice path matroids~\cite{Oh2011Positroids}. 

\begin{theorem}\label{thm:rook_matroids_positroids}
   The rook matroid $\rookMat_{\lambda/\mu}$ is a positroid.
\end{theorem}

\begin{proof}
    Let $\lambda/\mu$ be a skew shape with $r$ rows and $n-r$ columns. Define the $r \times n$ matrix $F = I_{r \times r} |  D_{\lambda/\mu}$ where $I_{r \times r}$ is an $r\times r$ identity matrix and $D_{\lambda/\mu}$ is the $r \times n-r$ matrix with columns indexed by $[r+1, n]$ and defined by \[
    (D_{\lambda/\mu})_{i, j} = \begin{cases}
        (-1)^{i+r}x_{i}^{j-1} \quad \text{if $(i, j) \in \lambda/\mu$}, \\
        0  \qquad \qquad \quad\text{otherwise.}
    \end{cases} 
    \]
    where the $x_{i}$ for $i \in [r]$ are positive reals whose  values we will specify later. Note that the support of $F$ (i.e. the set of non-zero entries) is exactly the same as that of the incidence matrix corresponding to the set system $A_{\lambda/\mu}$ defined after Fact~\ref{fact:skew_shape_characterization}. Let $B = L \sqcup J$, where $L = B \cap [r]$, $J = B \cap [r+1, n]$, and $|L| = \ell$. Let $\Delta_{B}$ be the minor of $F$ with columns from $B$. By the observation on the support of $F$, if $B$ is not a basis, then $\Delta_{B} = 0$; and if $B$ is a basis, then $\Delta_{B}$ has at least one non-zero term in its expansion as a determinant. We show that the term dominating this expansion is positive. Expanding $\Delta_{B}$ along the columns of $L$ (each of which is some standard basis vector) from left to right, we pick up a sign equal to $(-1)^{K}$, where $K = \sum_{a \in L}a + \frac{\ell(3-\ell)}{2}$. We then extract a factor of $(-1)^{b+r}$ for each $b \in [r] \setminus L$. We are then left with the minor of $D_{\lambda/\mu}$ with rows from $[r] \setminus L$ and columns from $J$. Since $B$ is a basis, the set $J$ can be matched to the rows $[r]\setminus L$ by a non-nesting rook placement. With the rows ordered increasingly and the columns of $J$ ordered decreasingly, this matching contributes the anti-diagonal monomial, so the anti-diagonal term of this minor is nonzero. By choosing $x_{i}$ such that $x_{r} > 1$ is sufficiently large, and $x_{r-i} = x_{r-i+1}^{r^{2}}$ for $i =1, \ldots , r-1$, we see that the anti-diagonal term, i.e., the term arising from the longest permutation on $r- \ell$ elements, strictly dominates the sum of the absolute values of all other non-zero terms. This anti-diagonal term has sign $(-1)^{\binom{r-\ell}{2}}$. Since the second term in the expression for $K$ above is equivalent mod $2$ to $\binom{\ell + 1}{2}$, the gathered powers of $-1$ reduce to
    \[
    \sum_{a \in L}a + \binom{\ell +1}{2} + \sum_{b \in [r] \setminus L}(b+r) + \binom{r-\ell}{2} =  2r(r-\ell) + \ell(\ell+1),
    \]
    which is even and hence the dominant term is positive and $\Delta_{B} > 0$.  
\end{proof}

\begin{example}\label{eg:rook_as_positroid}
    Let $\lambda/\mu = 4433/2$. The sets in the transversal presentation of $\rookMat_{\lambda/\mu}$ are $A_{1} = \{1,7,8\}$, $A_{2} = \{2,5,6,7,8\}$, $A_{3} = \{3,5,6,7\}$, $A_{4} = \{4,5,6,7\}$. Then $\rookMat_{\lambda/\mu}$ is realized by the following matrix $B$, all of whose maximal minors are non-negative:
    \[
B=\left[
\begin{array}{cccc|cccc}
1 & 0 & 0 & 0 & 0 & 0 & -x_{1}^{6} & -x_{1}^{7} \\
0 & 1 & 0 & 0 & x_{2}^{4} & x_{2}^{5} & x_{2}^{6} & x_{2}^{7} \\
0 & 0 & 1 & 0 & -x_{3}^{4} & -x_{3}^{5} & -x_{3}^{6} & 0 \\
0 & 0 & 0 & 1 & x_{4}^{4} & x_{4}^{5} & x_{4}^{6} & 0
\end{array}
\right].
\]
\end{example}

\subsection{Tutte polynomials}\label{subsection:tutte_shadow}

In this subsection, we show that the Tutte polynomials
of rook matroids and lattice path matroids coincide, thereby tightening the correspondence between these classes of matroids.

The \defin{Tutte polynomial} is defined as the corank-nullity generating function of the matroid $M$: \[
\defin{T_{M}(x, y)} \coloneqq \sum_{A \subseteq E}(x-1)^{r(E)-r(A)}(y-1)^{|A|-r(A)}.
\]
For an overview of results on Tutte polynomials, see \cite{Ellis2022HandbookTutte}. The important fact about the Tutte polynomial that we will use is that it satisfies the deletion-contraction recursion:
\begin{equation}
 T_{M}(x,y) = T_{M/e}(x,y) + T_{M\setminus e}(x,y),
\end{equation}
where $e \in E$ is neither a loop nor a coloop. 

We need the following lemma, the proof of which is simple and hence omitted.  

\begin{lemma}\label{lem:rookSimpleSymmetry}
 Let $M$ be a rook matroid on a skew shape with $r$ rows. Then the two deletions $M\setminus r$ and $M \setminus (r+1)$ are isomorphic, and thus have the same Tutte polynomial.
\end{lemma}

For the purpose of readability, we use $T(M; x, y)$ in place of $T_{M}(x, y)$ ahead. 

\begin{theorem}\label{thm:sameTutte}
Let $\lambda/\mu$ be a skew shape. 
Then $T(\pathMat_{\lambda/\mu}; x,y) = T(\rookMat_{\lambda/\mu};x,y)$.
\end{theorem}

\begin{proof}
We proceed by induction over the size of $\lambda / \mu$, the number of squares in the diagram of the skew shape. The base case is immediate since both the lattice path matroid and the rook matroid on a single square are equal to the uniform matroid $U_{1,2}$.

Assume that the statement holds for all diagrams with 
fewer boxes than in $\lambda/\mu$.
Let $D$ denote the diagram $\lambda/\mu$ and let us 
consider the following additional diagrams:
\[
\substack{
 \begin{ytableau}
 \none & \none & \none & \none & *(pBlue) & *(pBlue) & *(pBlue) \\
\none &    *(pBlue)    &     *(pBlue)  &  *(pBlue)     & *(pBlue) & *(pBlue) & *(pBlue) \\
 \none &    *(pBlue)   &    *(pBlue)   &    *(pBlue)   & *(pBlue) &  *(pBlue) \\
 \none &    *(pBlue)   &    *(pBlue)   &  *(pBlue) \\
 \end{ytableau}
 \\  D }
\quad 
\substack{
 \begin{ytableau}
 \none & \none & \none & \none & *(pBlue) & *(pBlue) & *(pBlue) \\
\none &         &     *(pBlue)  &  *(pBlue)     & *(pBlue) & *(pBlue) & *(pBlue) \\
 \none &        &    *(pBlue)   &    *(pBlue)   & *(pBlue) &  *(pBlue) \\
 \none &       &  *(pBlue)   &  *(pBlue) \\
 \end{ytableau}
 \\  D' }
\qquad 
\substack{
 \begin{ytableau}
\none &  \none & \none & \none & *(pBlue) & *(pBlue) & *(pBlue) \\
\none &    *(pBlue)    &     *(pBlue)  &  *(pBlue)     & *(pBlue) & *(pBlue) & *(pBlue) \\
 \none &    *(pBlue)   &    *(pBlue)   &    *(pBlue)   & *(pBlue) &  *(pBlue) \\
\none &        &     &   \\
 \end{ytableau}
 \\  D'' }
\quad 
\substack{
 \begin{ytableau}
  \none & \none & \none & \none & *(pBlue) & *(pBlue) & *(pBlue) \\
  \none &   *(pBlue)    &     *(pBlue)  &  *(pBlue)     & *(pBlue) & *(pBlue) & *(pBlue) \\
  \none &    *(pBlue)   &    *(pBlue)   &    *(pBlue)   & *(pBlue) &  *(pBlue) \\
  \none &    *(pAlgae)   &    *(pAlgae)   &  *(pAlgae) \\
 \end{ytableau}
 \\  D \setminus r }
\]
In $D'$ we have removed the first column of $D$, in $D''$ we have removed the last row of $D$, indexed by $r$. The last picture indicates that the last row contains a rook.

The Tutte polynomial for the lattice path matroid satisfies
\begin{equation}\label{eq:tutte_LPM}
  T(\pathMat_D;x,y) =  T(\pathMat_{D'};x,y) + T(\pathMat_{D''};x,y),
\end{equation}
by deletion-contraction, which corresponds to conditioning on whether or not the first step in the path is an East step or a North step.

Similarly, the Tutte polynomial for the rook matroid satisfies
\begin{equation}\label{eq:tutte_rook}
  T(\rookMat_D;x,y) =  T(\rookMat_{D} \setminus r;x,y) +  T(\rookMat_{D''};x,y),
\end{equation}
by considering deletion-contraction of the last row index.
By Lemma~\ref{lem:rookSimpleSymmetry}, 
we know that $T(\rookMat_{D}\setminus r;x,y) = T(\rookMat_{D} \setminus (r+1);x,y)$ = $T(\rookMat_{D'};x,y)$.
By the induction hypothesis, we have that 
\[
 T(\pathMat_{D''};x,y) = T(\rookMat_{D''};x,y) \text{ and }
 T(\pathMat_{D'};x,y) = T(\rookMat_{D'};x,y).
\]
Substituting the terms in the RHS of~(\ref{eq:tutte_rook}) with the expressions above and using~(\ref{eq:tutte_LPM}), it follows that $T(\pathMat_D;x,y) = T(\rookMat_D;x,y)$.
\end{proof}

 The Tutte polynomial is a valuative invariant of a matroid: it behaves well under subdivisions of matroid polytopes. On the basis of Theorem~\ref{thm:sameTutte} and computer experiments, we make the case that every valuative invariant has the same value on $\rookMat_{\lambda /\mu}$ as it does on $\pathMat_{\lambda /\mu}$. 

\begin{conjecture}\label{conjecture:sam_G_invariants}
Let $\lambda / \mu$ be a skew shape and let $\rookMat_{\lambda /\mu}$ and $\pathMat_{\lambda /\mu}$ be the corresponding rook matroid and lattice path matroid respectively. Then for every valuative invariant $f$ for matroids, we have \[
f(\rookMat_{\lambda /\mu}) = f(\pathMat_{\lambda /\mu}).
\]
\end{conjecture}

In particular, the same would hold for a universal valuative invariant like the Derksen--Fink $\mathcal{G}$-invariant  \cite{DerksenFink2010ValutativeInvariants}. For a systematic treatment of valuative invariants of matroids from a computational viewpoint, see \cite{FerroniSchroter2022Valuative}. Since our preprint first appeared, Conjecture~\ref{conjecture:sam_G_invariants} has been proved by Bonin and de Mier in \cite{BoninDeMier2025EqualConfig} using the concept of the configuration of a matroid.

\section{Distributional properties of non-nesting rook numbers}\label{section:distributional_properties}
In this section, we use the underlying matroid structure to deduce the distributional properties of the coefficients of the non-nesting rook polynomial. The main result of the following subsection can be seen as a non-nesting analog of the bipartite case of the Heilmann--Lieb theorem on the real-rootedness of the matching polynomial~\cite{HeilmannLieb1972, Nijenhuis1976}. For the definition of Lorentzian polynomials, see~\cite{BrandenHuh2020Lorentzian}.

\subsection{Ultra-log-concavity of \texorpdfstring{$M_{\lambda /\mu}$}{mmu}}
Let $\lambda / \mu$ be a skew shape with $r$ rows and $c$ columns. 
The basis-generating polynomial $r_{\lambda/\mu}$ of $\mathcal{R}_{\lambda / \mu}$ serves as the appropriate multivariate generalization of $M_{\lambda / \mu}$, the univariate non-nesting rook polynomial introduced in Section~\ref{subsection:boards_skew-shapes_rooks}. To emphasize the role played by column and row variables separately --- which will also be useful in Section~\ref{section:rooks_as_linear_extensions} --- we use $\xvec = (x_{1}, \ldots , x_{r})$ and $\yvec = (y_{r+1}, \ldots , y_{r+c})$ for the row and column variables of $r_{\lambda / \mu}$ respectively. By Theorem~\ref{thm:rook_is_transversal}, bases of rook matroids correspond to occupied column indices taken together with unoccupied row indices of rook placements; we can thus write

\begin{equation}
r_{\lambda / \mu}(\xvec, \yvec) = \sum_{\rho \in \NN_{\lambda / \mu}}\prod_{i \in R(\rho)}x_{i}\prod_{j \in C(\rho)}y_{j}\label{eq:multivariate_rook},
\end{equation}

where as before $\NN_{\lambda /\mu}$ denotes the set of non-nesting rook placements on $\lambda / \mu$ and $R(\rho), C(\rho)$ correspond to the set of row indices not occupied by $\rho$ and the set of column indices that are occupied by $\rho$ respectively.

\begin{corollary}\label{ultra_log_concavity_ncm}
    For every skew shape $\lambda / \mu$, the coefficient sequence of the non-nesting rook polynomial $M_{\lambda / \mu}(t) = \sum_{k=0}^{r}a_{k}(\lambda / \mu)t^{k} $ is  ultra-log-concave with no internal zeros. That is, \[
    \left(\dfrac{a_{k}}{\binom{r}{k}}\right)^{2} \geq \dfrac{a_{k-1}}{\binom{r}{k-1}}\cdot \dfrac{a_{k+1}}{\binom{r}{k+1}} \quad \text{for all $1\leq k \leq r-1$}.
    \]
Moreover, there exists a skew shape $\alpha/\beta$ such that $M_{\alpha/\beta}(t)$ is not real-rooted. 
\end{corollary}
\begin{proof}
Suppose $\lambda/\mu$ has $r$ rows and $c$ columns. Let $M$ be the rook matroid $\rookMat_{\lambda/\mu}$, $P_{M}$ be the corresponding basis-generating polynomial, and $T = [r+1, r+c]$. Consider $f(t, s) = P_{M}(t\cdot \mathbf{1}_{T} + s\cdot\mathbf{1}_{T^{c}})$ where $\mathbf{1}_T$ is the characteristic vector of $T$. Note that $P_{M}$ can be expressed as in~(\ref{eq:multivariate_rook}). The result then follows from the fact that $P_{M}$ is Lorentzian~\cite[Theorem 3.10]{BrandenHuh2020Lorentzian}, specializing variables preserves the Lorentzian property, and a homogeneous bivariate polynomial is  Lorentzian if and only if its coefficient sequence is ultra-log-concave with no internal zeros. 

The smallest counterexample to the real-rootedness of the non-nesting rook polynomial is given in Corollary~\ref{cor:nn_with_non_real_roots}, along with fuller context.
\end{proof}

\begin{remark}
Combining the corollary above with Proposition~\ref{lem:setBijectionPathRook}, it follows that the generating polynomial of the number of valleys ranging over the set of lattice paths contained inside a partition $\lambda$ is ultra-log-concave. Showing this using the matroidal structure of lattice path matroids alone appears to be difficult. 
\end{remark}

At this juncture, we note the parallel between the non-nesting rook polynomial and the full rook polynomial: both are ultra-log-concave. We will see that this is essentially the strongest property we can hope for, as it turns out that unlike the full rook polynomial the non-nesting counterpart is \textit{not} real-rooted. In Section~\ref{subsection:matroidal_lifts}, the proper context of this failure of real-rootedness will be given.

\subsection{Lattice path matroids are not HPP}\label{subsection:LPM_not_HPP}
In the previous subsection, we saw that ultra-log-concavity holds for the non-nesting rook numbers. In this subsection, we consider a strengthening of real-rootedness (stability) and give the first explicit example of a lattice path matroid that does not have the half-plane property. For the definitions and overview of stable polynomials in the context of matroid theory, see~\cite{Branden2015}. 

A robust necessary condition that a homogeneous, multiaffine polynomial is stable is that its support forms the collection of bases of a matroid \cite{ChoeOxleySokalWagner2004}. The \defin{half-plane property} (HPP) of a matroid $M$ identifies whether the converse is true, namely that the basis-generating polynomial of $M = (E, \mathcal{B})$ defined as
\[
P_{M}(\xvec) = \sum_{B \in  \mathcal{B}}\prod_{i \in B}x_{i},
\]
is stable. We call such a matroid HPP. In \cite{ChoeOxleySokalWagner2004}, an in-depth investigation of the half-plane property for matroids was carried out. One question raised -- a ``wild speculation'' according to \cite{ChoeOxleySokalWagner2004} --  was whether transversal matroids have the HPP. Shortly thereafter, Choe and Wagner provided an example of a rank $4$ transversal matroid on $12$ elements that was not HPP \cite{ChoeWagner2006}. Their example, however, is not a lattice path matroid; this can be straightforwardly checked using a criterion described in \cite[Theorem 3.14]{Bonin2006lattice}.
To our knowledge, the question of whether lattice path matroids have the HPP is thus still open. The purpose of this subsection is to give the first explicit example of a lattice path matroid---in fact a Catalan matroid---that does not have the half-plane property. 

\begin{theorem}\label{thm:LPM_not_HPP}
    There exists a generalized Catalan matroid that is not HPP. Concretely, the basis-generating polynomial of the matroid $\pathMat_{\lambda}$ for $\lambda = 666333$ is not stable. 
\end{theorem}

\begin{proof}
Let $\lambda = 666333$ and $P(\xvec)$ be the basis-generating polynomial of $\pathMat_{\lambda}$. Consider the univariate polynomial $f(t)$ obtained by the substitution $x_{i} =t$ for $i=1, \ldots , 6$ and $x_{j}=1$, for $j=7, \ldots,12$ in $P(\xvec)$. Using SageMath~\cite{SageMath}, we found that \[f(t) = t^{6}+36t^{5}+225t^4 + 400t^{3}.\]
    This polynomial has non-real roots near $-3.6855 \pm 0.6232i.$ Since specializing variables in stable polynomials to real values preserves stability, it follows that $P(\xvec)$ is not stable.
\end{proof}

\begin{corollary}\label{cor:Catalan_M10_not_HPP}
    The Catalan matroid $M_{10}$ is not HPP. 
    That is, the basis-generating polynomial of $M_{10}$ is not stable.
\end{corollary}

\begin{proof}
 Every generalized Catalan matroid is a minor of some Catalan matroid \cite[Theorem 4.2]{Bonin2006lattice}. In particular, if $M_{10}$ denotes the Catalan matroid of order $10$ (defined as the lattice path matroid on the Ferrers shape $\delta_{10}$), then $\pathMat_{666333}$ is a minor of $M_{10}$. Since having the half-plane property is closed under taking minors~\cite{ChoeOxleySokalWagner2004}, it follows that $M_{10}$ cannot have the half-plane property. 
\end{proof}
\begin{remark}
In \cite{Xu2015} it is stated that the class 
of generalized Catalan matroids is not HPP, 
however, no explicit counterexample is given.
\end{remark}

\section{Rook placements as linear extensions}\label{section:rooks_as_linear_extensions}

In this section, we show a bijective correspondence between labeled skew shapes and posets of width two. We obtain an interpretation for the non-nesting rook polynomial in terms of a $P$-Eulerian polynomial of a poset and deduce the following results. The $P$-Eulerian polynomial of a naturally labeled width two poset is ultra-log-concave, thereby completing the picture of the Neggers--Stanley conjecture, which was disproved using a naturally labeled width two counterexample \cite{Stembridge2007NeggersStanleyCounter}.  We deduce this from a stronger statement, namely that a suitable multivariate analog of the $P$-Eulerian polynomial---analogous to the one considered in \cite{BrandenLeander2016x}---is Lorentzian. Returning to the rook polynomial setting, we learn that there exists a skew shape $\lambda / \mu$ for which the non-nesting rook polynomial $M_{\lambda /\mu}$ has non-real roots. We end by considering the gamma-positivity properties of  $M_{\lambda/\mu}$.

\subsection{Neggers--Stanley conjecture}
In this subsection we recall the statement and context of the Neggers--Stanley conjecture. For undefined terminology concerning posets, we refer to \cite{StanleyEC1}. 

 Recall that given a poset $P$, the \defin{width} of $P$ is the size of the largest antichain in $P$. A labeling of a poset $P$ on $n$ elements is a bijection $\omega: P \to [n]$. We say that $\omega$ is a \defin{natural labeling} if $i \prec j$ implies $\omega(i) < \omega(j).$  The \defin{Jordan-H\"{o}lder set} of $(P, \omega)$ is the set of all permutations of $[n]$ the inverses of which are linear extensions of $(P, \omega)$ i.e., it is defined as \[
\mathcal{L}(P, \omega) = \{\sigma \in \mathfrak{S}_{n}: i \prec j \implies \sigma^{-1}({\omega(i)}) < \sigma^{-1}({\omega(j)})\}.
\]
In other words, $\sigma \in \mathcal{L}(P)$ if for every relation $i \prec j$ in $P$, we have that $\omega(i)$ precedes $\omega(j)$ in the one-line representation of the permutation $\sigma.$

The \defin{$(P, \omega)$-Eulerian polynomial}, also known as the \defin{$W$-polynomial of $P$}, is the descent-generating polynomial of the Jordan-H\"{o}lder set of $(P, \omega)$: \[
W_{P, \omega}(t) = \sum_{\sigma \in \mathcal{L}(P, \omega)}t^{\des(\sigma)}.
\]

When $\omega$ is natural, we simply write $W_{P}$, since the polynomial is independent of the choice of natural labeling. Observe that in the definitions of the multivariate analog that follow, there is a dependence on the choice of natural labeling that we make explicit. 

The distributional properties of $W_{P, \omega}$ were of early interest to combinatorialists working in poset theory. The following conjecture was first formulated by Neggers in 1978 for natural labelings $\omega$; in 1986, Stanley extended it to arbitrary labelings. Subsequent references to this conjecture also called it the Poset conjecture \cite{Brenti1989UnimodalAMS}. For further background, the reader can consult \cite[Section 6]{Branden2015}.

\begin{conjecture}[Neggers--Stanley conjecture \cite{Neggers1978, Stanley1986}]\label{conjecture:neggers_stanley}
Let $(P, \omega)$ be a labeled poset. Then $W_{P, 
  \omega}$ is real-rooted.
\end{conjecture}

The Neggers--Stanley conjecture was of central importance to algebraic combinatorics until its resolution in the negative in the early aughts: first by Br\"{a}nd\'{e}n \cite{Branden2004NeggersStanley} who found a family of counterexamples to Stanley's formulation and then Stembridge \cite{Stembridge2007NeggersStanleyCounter} who disproved Neggers' counterpart as well. In both cases, the counterexample furnished was of a width two poset; Br\"{a}nd\'{e}n's construction was non-naturally labeled while Stembridge's (larger) counterexample was naturally labeled. Despite this breakthrough, the question of unimodality or log-concavity of $W_{P, \omega}$ for general $(P, \omega)$ remained open. In particular, the following conjecture of Brenti has been open since 1989. 

\begin{conjecture}\label{conjecture:brenti1989conjecture}\cite[Conjecture 1.1]{Brenti1989UnimodalAMS}
Let $(P, \omega)$ be a labeled poset. Then $W_{P, 
  \omega}$ is log-concave with no internal zeros.
\end{conjecture}

There have been two positive results towards this end. Reiner and Welker proved that when $P$ is naturally labeled and graded, $W_{P}$ is unimodal and symmetric \cite{ReinerWelker2005CharneyDavis}. Br\"{a}nd\'{e}n then gave an elegant combinatorial proof of a stronger property: namely that $W_{P}$ is $\gamma$-positive for the larger class of sign-graded posets \cite{Branden06SignGraded, Branden2008ActionsOnPermutations}.

 In the next subsection, we show a strengthening of  Conjecture~\ref{conjecture:brenti1989conjecture} for special posets $P$. Namely, we show that when $P$ is naturally labeled and of width two, then $W_{P}$ is ultra-log-concave. We do so by formulating an appropriate multivariate analog of $W_{P}$ and recognizing it as the basis-generating polynomial of a rook matroid.

\subsection{Matroidal lifts of \texorpdfstring{$P$}{P}-Eulerian polynomials}\label{subsection:matroidal_lifts}

Given a skew shape $\lambda / \mu$, our goal is to obtain a suitable poset $P$ 
such that the $W$-polynomial of $P$ is equal to the non-nesting 
rook polynomial of $\lambda / \mu$. 
We do so in Theorem~\ref{thm:path_to_poset}. 
Before we state this, we introduce a multivariate analog of $W_{P}$. Our choice is distinct from, but inspired by, the definition in~\cite{BrandenLeander2016x}.

Let $(P, \omega)$ be a naturally labeled poset of 
width two. Fix a chain decomposition $C_{1}\sqcup C_{2}$ of $P$ where the chain $C_{1}$ has $r$ elements and 
the chain $C_{2}$ has $c$ elements. We emphasize that the multivariate polynomial $\widetilde{W}_{P, \omega}$ defined ahead depends on the choice of chain decomposition of $P$, but the univariate polynomial $W_{P, \omega}$ does not. (The chain decomposition of $(P, \omega)$ is unique when $P$ is irreducible~\cite[Proposition 5.1]{Stembridge2007NeggersStanleyCounter}.) With respect to this chain decomposition, define 
 $\widetilde{W}_{P, \omega}\in \mathbb{N}[x_{e},y_{e'}: e \in C_{1},e' \in C_{2}]$ as 
\begin{equation}
    \widetilde{W}_{P, \omega}(\xvec, \yvec) = \sum_{\sigma \in \mathcal{L}(P, \omega)}\prod_{i \in \text{RA}(\sigma)}x_{i}\prod_{j \in \text{DT}(\sigma)}y_{j},
\end{equation}
where \begin{enumerate}
    \item $\text{RA}(\sigma) = \{\sigma_{i} \in [r+c]: \sigma_{i} > \sigma_{i-1}, \sigma_{i} \in C_{1}\}$ is the set of \defin{row ascents} of $\sigma$, i.e. the set of ascent tops lying in $C_{1}$ and $\sigma_{0} = - \infty$.
        \item $\text{DT}(\sigma) = \{\sigma_{i} \in [r+c]: \sigma_{i}>\sigma_{i+1}\}$ is the set of \defin{descent tops}.
\end{enumerate}

The condition in the definition of $\text{RA}(\sigma)$ implies that the element $\sigma_{1}$ is an 
ascent top if and only if $\sigma_{1} \in C_{1}$, which by virtue of $C_{1}$ being a chain with minimum element $1$, is equivalent to $\sigma_{1} = 1$. Here $\mathrm{RA}$ stands for row ascent, for a reason that will be made clear during the course of the proof of Theorem~\ref{thm:path_to_poset}. Also note that $\widetilde{W}_{P, \omega}$ is homogeneous of degree $|C_{1}|$. It is not necessarily true that for every natural labeling $\omega$, this polynomial has only $0/1$ coefficients; this holds, however, in the example below. 

\begin{example}\label{eg:W_tilde_example}
Suppose $(P, \omega)$ is the labeled poset on the left below. Here, $P$ is irreducible (i.e. it cannot be written as the ordinal sum of two subposets) with chains $C_{1} = \{1,2,4\}$ and $C_{2} = \{3,5\}$. The Jordan--Hölder set is $\mathcal{L}(P, \omega) = \{ 31254,
 31245,
 13245, 12345, 12354,
 13254,
 13524,
 31524\}$. The polynomial $\widetilde{W}_{P, \omega}$ 
is a polynomial in the variables $x_{1}$, $x_{2}$, $x_{4}$, $y_{3}$, $y_{5}$ and 
is equal to
\[
\begin{tikzpicture}[baseline=(current bounding box.center)]
\node[dot=4pt, fill=black, label={left:{$1$}},label={[xshift=0cm, yshift=-0.85cm, text=black]{$C_{1}$}}] at (0, 0) (a1) {};
\node[dot=4pt, fill=black, label={left:{$2$}}] at (0, 2) (a2) {};
\node[dot=4pt, fill=black, label={left:{$4$}}] at (0, 4) (a3) {};

\node[dot=4pt, fill=black, label={right:{$3$}},label={[xshift=0cm, yshift=-0.85cm, text=black]{$C_{2}$}}] at (4, 0) (a6) {};
\node[dot=4pt, fill=black, label={right:{$5$}}] at (4, 2) (a7) {};

\draw[very thick,black] (a1)-- (a2) -- (a3);
\draw[very thick,black] (a6)-- (a7);
\draw[very thick,black] (a6)-- (a3);
\draw[very thick,black] (a1) -- (a7);
\end{tikzpicture}
\qquad
\text{$\widetilde{W}_{P, \omega}$ } 
= x_{2} y_{3} y_{5} + x_{2} x_{4} y_{3} + x_{1} x_{4} y_{3} + x_{1} x_{2} x_{4} + x_{1} x_{2} y_{5} + x_{1} y_{3} y_{5} + x_{1} x_{4} y_{5} + x_{4} y_{3} y_{5}.
\]

\end{example}

Recall that if $r_{\lambda / \mu}$ is the basis-generating polynomial of the rook matroid on $\lambda/\mu$, then we can express it in terms of row variables $x_{1}, \ldots , x_{r}$ and column variables $y_{r+1}, \ldots , y_{r+c}$ as \[
r_{\lambda/\mu}(\xvec, \yvec) = \sum_{\rho \in \NN_{\lambda /\mu}} \prod_{i \in R(\rho)}x_{i}\prod_{j \in C(\rho)}y_{j}.
\]

We can now state the main theorem of the section.

\begin{theorem}\label{thm:path_to_poset}
    Let $\lambda / \mu$ be a skew shape and $r_{\lambda / \mu}$ be the basis-generating polynomial of the rook matroid on $\lambda / \mu$. There exists a naturally labeled poset $(P, \omega)$ of width two such that 
    \begin{equation}\label{eq:rook_as_w_polynomial} 
    r_{\lambda / \mu}(\xvec, \yvec) \cong  \widetilde{W}_{P, \omega}(\xvec, \yvec) 
    \end{equation}
    where $\cong$ denotes equality up to reindexing of the variables.
\end{theorem}
\begin{remark}
    The fact that linear extensions of posets of width $k$ bijectively correspond to lattice paths contained within a compact region of $\mathbb{R}^{k}$ follows from the standard theory of distributive lattices (see~\cite{Dilworth1950DecompositionTheorem} or~\cite[pg. 296]{StanleyEC1}, for example). Similar bijections to ours appear in \cite[Lemma 8.1]{ChanPakPanova2022CrossProdConj} and an unpublished paper of Stanley \cite{Stanley2023WidthTwo}. The theorem above can be seen as a matroidal and multivariate refinement of the aforementioned results.
\end{remark}

\begin{proof}[Proof of Theorem~\ref{thm:path_to_poset}]
We begin by relabeling the rows and columns of $\lambda / \mu$ in the following manner. 
Denote the innermost lattice path of the skew shape by $L_{\lambda /\mu}$ and number the 
rows of $\lambda / \mu$ from bottom to top and the columns from left to right by the North 
and East steps of $L_{\lambda /\mu}$ respectively. Every reference to a cell $(i, j)$ of the skew shape is with respect to this labeling. See Figure \ref{fig:narayana_snake_LPM} 
for an example. Let $C_{1}$ and $C_{2}$ be the two chains defining $P$, labeled 
respectively by the row and column labels of $\lambda / \mu$ that were induced by $L_{\lambda /\mu}$. Let this labeling be $\omega$. 
Hereafter, in the context of cover relations, we will use $i, j$ in place of 
$\omega(i), \omega(j)$.

Add the cover relation $i \precdot j$ to $P$ for every outer corner $(i, j)$ 
of $\lambda / \mu$ and similarly the cover relation $j \precdot i$ for 
every inner corner $(i, j)$. This construction is illustrated in 
Figure~\ref{fig:poset_from_skew_shape}. We make the following claim:

\textbf{Claim 1:} The poset $(P, \omega)$ is naturally labeled and of width two. 

\textbf{Proof of claim 1:} The width two observation is immediate from the construction of $P$. To see that $P$ is naturally labeled, we make an intermediate claim: for every cell $(i, j) \in  \lambda / \mu$ (with the above labeling), $i < j.$ Since the row labels decrease from top to bottom, it suffices to show this for a cell $(i, j)$ that is the topmost box in its column. This holds since every topmost box in a column lies directly under or to the right of a $\mathtt{NE}$ segment of the innermost path $L_{\lambda/\mu}$.

Now consider the cover relation $a \precdot b$ where $a \in C_{1}, b \in C_{2}$. This implies that $(a, b)$ is an 
outer corner of the shape. The cell to 
the immediate left of $(a, b)$, say $(a, c)$, must lie in $\lambda / \mu$ and 
hence $a<c$, by the intermediate claim. Since the column labeling increases 
from left to right, we must also have $a<b$. The inner corner case is similarly proved, which finishes the proof of the claim.

\begin{figure}[htbp]
\begin{subfigure}[b]{0.30\textwidth}
\centering
\begin{tikzpicture}[inner sep=0in,outer sep=0in]
\node (n) {\begin{varwidth}{6cm}{
\ytableausetup{boxsize=1.25em}
\begin{ytableau} \none & \none[4] & \none[6] & \none[7] & \none[9] & \none[10]  \\ \none[8] & \none & \none & \none &  &  \\ \none[5] & \none &  &  & \\ \none[3] & & \none &  &  \\ \none[2] &  &  \\ \none[1] &  \\ \end{ytableau}}\end{varwidth}};
\coordinate (a1) at ([xshift=1.30em]n.south west);
%\node(A1)[] at (a1) {$\bullet$};
\coordinate (a2) at ($(a1)+( 0, 1.30em)$);
\coordinate (a3) at ($(a2)+(0, 1.3em)$);
\coordinate (a4) at ($(a3)+(0, 1.30em)$);
\coordinate (a5) at ($(a4)+(1.30em, 0)$);
\coordinate (a6) at ($(a5)+(0, 1.30)$);
\coordinate (a7) at ($(a6)+(1.30em, 0)$);
\coordinate (a8) at ($(a7)+(1.30em, 0)$);
\coordinate (a9) at ($(a8)+(0, 1.30em)$);
\coordinate (a10) at ($(a9)+(1.30em, 0)$);
\coordinate (a11) at ($(a10)+(1.30em, 0)$);

\coordinate (b1) at ($(a2)+(1.30em, 0)$);
\coordinate (b2) at ($(a4)+(1.30em, 0)$);
\coordinate (b3) at ($(a6)+(1.30em, 0)$);
\coordinate (b4) at ($(a8)+(1.30em, 0)$);

\draw[very thick,dashed] (a1)-- (a2) -- (a3)-- (a4)-- (a5)-- (a6)-- (a7)-- (a8)-- (a9)-- (a10)--(a11);
\coordinate (c1) at ([xshift=1.30em]n.south west);

\coordinate (c2) at ($(c1)+(1.30em, 0)$);
\coordinate (c3) at ($(c2)+(0, 1.30em)$);
\coordinate (c4) at ($(c3)+(0, 1.30em)$);
\coordinate (c5) at ($(c4)+(1.30em, 0)$);
\coordinate (c6) at ($(c5)+(0, 1.30em)$);
\coordinate (c7) at ($(c6)+(1.30em, 0)$);
\coordinate (c8) at ($(c7)+(0, 1.30em)$);
\coordinate (c9) at ($(c8)+(1.30em, 0)$);
\coordinate (c10) at ($(c9)+(1.30em, 0)$);
\coordinate (c11) at ($(c10)+(0, 1.30em)$);
\node[dot=4pt, fill=red] at (a5) {};
\node[dot=4pt, fill=red] at (a8) {};
\node[dot=4pt, fill=blue] at (c3) {};
\node[dot=4pt, fill=blue] at (c5) {};
\node[dot= 4pt, fill=blue] at (c9) {};

\end{tikzpicture}
\caption{Skew shape with innermost path marked with dashes.}
\label{fig:narayana_snake_LPM}
\end{subfigure}
~
\begin{subfigure}[b]{0.30 \textwidth}
\centering
\begin{tikzpicture}
\node[dot=4pt, fill=black, label={left:{$a_{1}$}}, label={[xshift=-0.75cm, yshift=-0.25cm, text=blue]{$1$}},label={[xshift=0cm, yshift=-0.85cm, text=black]{$C_{1}$}}] at (0, 0) (a1) {};
\node[dot=4pt, fill=black, label={left:{$a_{2}$}}, label={[xshift=-0.75cm, yshift=-0.25cm, text=blue]{$2$}}] at (0, 2) (a2) {};
\node[dot=4pt, fill=black, label={left:{$a_{3}$}}, label={[xshift=-0.75cm, yshift=-0.25cm, text=blue]{$3$}}] at (0, 4) (a3) {};
\node[dot=4pt, fill=black, label={left:{$a_{4}$}}, label={[xshift=-0.75cm, yshift=-0.25cm, text=blue]{$5$}}] at (0, 6) (a4) {};
\node[dot=4pt, fill=black, label={left:{$a_{5}$}}, label={[xshift=-0.75cm, yshift=-0.25cm, text=blue]{$8$}}] at (0, 8) (a5) {};

\node[dot=4pt, fill=black, label={right:{$a_{6}$}}, label={[xshift=0.75cm, yshift=-0.25cm, text=blue]{$4$}},label={[xshift=0cm, yshift=-0.85cm, text=black]{$C_{2}$}}] at (4, 0) (a6) {};
\node[dot=4pt, fill=black, label={right:{$a_{7}$}}, label={[xshift=0.75cm, yshift=-0.25cm, text=blue]{$6$}}] at (4, 2) (a7) {};
\node[dot=4pt, fill=black, label={right:{$a_{8}$}}, label={[xshift=0.75cm, yshift=-0.25cm, text=blue]{$7$}}] at (4, 4) (a8) {};
\node[dot=4pt, fill=black, label={right:{$a_{9}$}}, label={[xshift=0.75cm, yshift=-0.25cm, text=blue]{$9$}}] at (4, 6) (a9) {};
\node[dot=4pt, fill=black, label={right:{$a_{10}$}}, label={[xshift=0.85cm, yshift=-0.25cm, text=blue]{$10$}}] at (4, 8) (a10) {};

\draw[very thick,black] (a1)-- (a2) -- (a3)-- (a4)-- (a5);
\draw[very thick,black] (a6)-- (a7)-- (a8)-- (a9)-- (a10);
\draw[very thick,blue] (a1)-- (a7);
\draw[very thick,blue] (a2)-- (a8);

\draw[very thick,blue] (a4)-- (a10);

\draw[very thick,red] (a6)-- (a4);
\draw[very thick,red] (a8)-- (a5);
\end{tikzpicture}
\caption{Width two $(P, \omega)$ corresponding to skew shape. }
  \label{fig:poset_from_skew_shape}
\end{subfigure}
~
\begin{subfigure}[b]{0.30\textwidth}
     \begin{tikzpicture}[inner sep=0in,outer sep=0in]
\node (n) {\begin{varwidth}{6cm}{
\ytableausetup{boxsize=1.25em}
\begin{ytableau} \none & \none[4] & \none[6] & \none[7] & \none[9] & \none[10]  \\ \none[8] & \none & \none & \none &  & \rook \\ \none[5] & \none &  &  \rook & \\ \none[3] &  &  \rook &  &  \\ \none[2] &  &  \\ \none[1] &  \rook \\ \end{ytableau}}\end{varwidth}};
\coordinate (a1) at ([xshift=1.30em]n.south west);

\coordinate (a2) at ($(a1)+( 0, 1.30em)$);
\coordinate (a3) at ($(a2)+(0, 1.3em)$);
\coordinate (a4) at ($(a3)+(0, 1.30em)$);
\coordinate (a5) at ($(a4)+(1.30em, 0)$);
\coordinate (a6) at ($(a5)+(0, 1.30em)$);
\coordinate (a7) at ($(a6)+(1.30em, 0)$);
\coordinate (a8) at ($(a7)+(1.30em, 0)$);
\coordinate (a9) at ($(a8)+(0, 1.30em)$);
\coordinate (a10) at ($(a9)+(1.30em, 0)$);
\coordinate (a11) at ($(a10)+(1.30em, 0)$);

\coordinate (b1) at ($(a2)+(1.30em, 0)$);
\coordinate (b2) at ($(a4)+(1.30em, 0)$);
\coordinate (b3) at ($(a6)+(1.30em, 0)$);
\coordinate (b4) at ($(a8)+(1.30em, 0)$);

\draw[very thick,dashed] (a1)-- (a2) -- (a3)-- (a4)-- (a5)-- (a6)-- (a7)-- (a8)-- (a9)-- (a10)--(a11);

\coordinate (c1) at ([xshift=1.30em]n.south west);

\coordinate (c2) at ($(c1)+(1.30em, 0)$);
\coordinate (c3) at ($(c2)+(0, 1.30em)$);
\coordinate (c4) at ($(c3)+(0, 1.30em)$);
\coordinate (c5) at ($(c4)+(1.30em, 0)$);
\coordinate (c6) at ($(c5)+(0, 1.30em)$);
\coordinate (c7) at ($(c6)+(1.30em, 0)$);
\coordinate (c8) at ($(c7)+(0, 1.30em)$);
\coordinate (c9) at ($(c8)+(1.30em, 0)$);
\coordinate (c10) at ($(c9)+(1.30em, 0)$);
\coordinate (c11) at ($(c10)+(0, 1.30em)$);

\draw[ultra thick, red] (c1)-- (c2) -- (c3)-- (c4)-- (c5)-- (c6)-- (c7)-- (c8)-- (c9)-- (c10)--(c11);
\end{tikzpicture}
\caption{Rook placement with lattice path in red.}
    \label{fig:rook_to_lattice_path_linear_ext}
\end{subfigure}
\caption{Skew shape -- poset correspondence. Indices of North and East steps of the dashed path in (A) form two disjoint chains, $C_{1}$ and $C_{2}$ of the poset in (B). Blue and red cover relations correspond to outer and inner corners respectively. The path permutation corresponding to the path in (C) is $\sigma = 41263759 \, 10 \, 8$. Here $\text{DT}(\sigma) = \{4,6,7,10\}$ and $\text{RA}(\sigma) = \{2\}$. The set $\text{DT}(\sigma) \cup \text{RA}(\sigma)$ represents the rook placement in the picture above.}
\label{fig:skew_to_poset}
\end{figure}

Now we exhibit a bijection from non-nesting rook placements inside $\lambda / \mu$ to elements of the Jordan-Hölder set $\mathcal{L}(P, \omega)$. This bijection will send non-nesting rook configurations with $k$ occupied rows to inverses of linear extensions with $k$ descents. Formally, define $f: \NN_{\lambda / \mu} \to \mathcal{L}(P, \omega)$ by $f = \psi\circ T_{\lambda / \mu}$ where $T_{\lambda / \mu}: \NN_{\lambda / \mu} \to \pathMat_{\lambda / \mu}$ is the bijection between rook placements and lattice paths described in Proposition~\ref{lem:setBijectionPathRook} and $\psi: \pathMat_{\lambda / \mu} \to \mathcal{L}(P, \omega)$ is the path-to-permutation map  $\psi(L) = \sigma = \sigma_{1}\ldots \sigma_{r+c}$, constructed in Definition~\ref{def:path_permutation}. An illustration of this is in Figure~\ref{fig:rook_to_lattice_path_linear_ext}.

\textbf{Claim 2:} The map $f$ defined above is a bijection. 

\textbf{Proof of Claim 2:} By the uniqueness of the permutation representation of the lattice path and the rook-to-path bijection from Proposition~\ref{lem:setBijectionPathRook}, the map $f$ is well-defined. To show that $f$ is a bijection, it suffices to show that $\psi$ is a bijection.

To that end, we first need to show that $\sigma = \psi(L)$ does indeed lie in $\mathcal{L}(P, \omega)$, whenever $L$ is a lattice path contained inside $\lambda/\mu$. Suppose $i \precdot j$ in $P$. Two cases arise: 

\underline{Case 1:} Both $i, j$ lie in the same chain of $P$, and are hence either both row indices or both column indices; since row indices increase from top to bottom and column indices increase from left to right in the labeling of $\lambda / \mu$, the path $L$ must traverse step $\sigma^{-1}(i)$ before step $\sigma^{-1}(j)$ and hence $i$ must precede $j$ in $\sigma$. 

\underline{Case 2:} The elements $i, j$ lie in different chains. Without loss of generality, assume that $i$ is a row index and $j$ is a column index, so that $(i, j)$ is an outer corner of $\lambda/\mu$. Then the $(\sigma^{-1}(i))^\thsup$ and $(\sigma^{-1}(j))^\thsup$ steps of $L$ are respectively North and East steps. Let $Q$ be the lower boundary path of $\lambda/\mu$. Since $(i,j)$ is an outer corner, $Q$ has a North-East segment at this corner. Let $(\psi(Q))_k=i$, so that $(\psi(Q))_{k+1}=j$. Since $L$ is contained in $\lambda/\mu$, it lies weakly above $Q$. Hence the North step of $L$ in row $i$ occurs no later than the corresponding North step of $Q$, while the East step of $L$ in column $j$ occurs no earlier than the corresponding East step of $Q$. This implies that $\sigma$ satisfies $\sigma^{-1}(i) \leq k < k+1 \leq \sigma^{-1}(j)$, as required. The case where $i,j$ are respectively column and row indices is analogous, using the upper boundary path of $\lambda/\mu$. This shows that the image of $f$ lies in $\mathcal{L}(P, \omega).$ 

To see that $\psi$ is a surjection, consider $\sigma \in \mathcal{L}(P, \omega)$. Define a lattice path $L$ by letting the $i^\thsup$ step of $L$ be North if $\sigma_{i}$ is in $C_{1}$ and East if $\sigma_{i}$ is in $C_{2}$; such a path certainly satisfies $\psi(L) = \sigma$ once we verify that $L$ is contained inside the shape. To do so, suppose $L$ exits the skew shape via the lower boundary of $\lambda / \mu$. Then there exists an outer corner $(i, j)$ of $\lambda / \mu$ such that $ji$ forms an $\mathtt{EN}$ step of $L$. In particular $\sigma = \sigma_{1}\ldots ji \ldots \sigma_{r+c}$. Now $(i, j)$ is an outer corner which implies that $i\precdot j$ and since $\sigma \in \mathcal{L}(P, \omega)$, $i$ must precede $j$ in $\sigma$, a contradiction. The case where $L$ exits through the upper boundary is analogous: one obtains an inner corner $(i,j)$ such that $ij$ forms an $\mathtt{NE}$ segment of $L$, contradicting the relation $j\precdot i$. Thus $L$ is a lattice path contained within $\lambda / \mu$ and hence $\psi$ is a surjection, and $f$ is indeed a bijection.

\textbf{Claim 3:} Let $\rho$ be the rook placement that corresponds to $\sigma$ under the bijection $f$. Then \[C(\rho) = \text{DT}(\sigma) \quad \text{and} \quad R(\rho) = \text{RA}(\sigma).\] That is, under the innermost path labeling of $\lambda/\mu$, occupied columns of $\rho$ and the unoccupied rows of $\rho$ are precisely equal to the descent tops of $\sigma$ and row ascents of $\sigma$ respectively.

\textbf{Proof of Claim 3:} Let $T_{\lambda /\mu}(\rho)$ be the lattice path $L_{\rho} = L$ and let $\sigma_L$ be the path permutation of $L$, with the convention that $\sigma_{0} = -\infty$. Then 
\[
C(\rho)= \psi(\{\text{East steps of inner valleys of  $L$}\}) \\ = \{\sigma^{L}_{i} \in C_{2}: \sigma^{L}_{i} > \sigma^{L}_{i+1} \} \\ = \text{DT}(\sigma), \quad \text{and}\]
\[
R(\rho) = C_{1}\setminus \psi(\{\text{North steps of inner valleys of $L$}\}) = \{\sigma^{L}_{i} \in C_{1}: \sigma^{L}_{i-1} < \sigma^{L}_{i}\} = \text{RA}(\sigma).\]

This completes the proof of the assertion that $\widetilde{W}_{P, \omega}$ can be written as a multivariate generating polynomial of non-nesting rook placements on $\lambda /\mu$. To finish, we note that after reindexing the rows and columns of the skew shape in accordance with the matroid labeling, this latter polynomial is precisely the basis-generating polynomial of $\rookMat_{\lambda / \mu}$.
\end{proof}

Let $\mathcal{E}$ be the map from Theorem~\ref{thm:path_to_poset} that takes a skew shape $\lambda/\mu$ to a width two poset $P$.
From the construction in the proof above, it is implicit that $\mathcal{E}$ is injective. We record the surjectivity of $\mathcal{E}$ in the lemma below. 

\begin{lemma}\label{lem:poset_to_path} 
    For every poset $P$ of width two, there exists a skew shape $\lambda/\mu$ such that $\mathcal{E}(\lambda/\mu) = P$.
\end{lemma}

\begin{proof}
First we repeatedly delete unique minimal and unique maximal elements of $P$, whenever
they exist, until no such element remains. (This reduction is harmless since the skew shape corresponding to the non-reduced poset is obtained from that of the reduced one by prepending empty rows and empty columns.) Once the poset --- which by a minor abuse of notation we continue to denote by $P$ --- has two minimal and two maximal elements, proceed as follows.

Let $\{t_{1}, \ldots , t_{r+c}\}$ be the elements of $P$ and let $C_{1}, C_{2}$ be the two chains of $P$ of length $r$ and $c$ respectively. Define $\lambda/\mu$ to be the skew shape with $r$ rows (labeled from bottom to top by elements of $C_{1}$) and $c$ columns (labeled from left to right by elements of $C_{2}$) such that: \begin{enumerate}
    \item There is an inner corner at $(t_{i}, t_{j})$ for every cover relation of the form $t_{i} \precdot t_{j}$ where $t_{i} \in C_{2}$ and $t_{j} \in C_{1}$. 
    \item There is an outer corner at $(t_{i}, t_{j})$ for every cover relation of the form $t_{i} \precdot t_{j}$ where $t_{i} \in C_{1}$ and $t_{j} \in C_{2}$.
\end{enumerate}
By the correspondence between corners and cover relations detailed in the proof of Theorem~\ref{thm:path_to_poset}, this is precisely the skew shape $\lambda/\mu$ such that $\mathcal{E}(\lambda/\mu) = P$.
\end{proof}
\begin{remark}
    Let $(P, \omega')$ be a naturally labeled poset. While Theorem~\ref{thm:path_to_poset} and  Lemma~\ref{lem:poset_to_path} together yield a skew shape $\lambda/\mu$ such that $(\mathcal{E}(\lambda / \mu), \omega)$ is a naturally labeled poset satisfying $\mathcal{E}(\lambda /\mu) = P$, the two  natural labelings $\omega, \omega'$ might differ. This will not matter for us, however, since the polynomial $W_{P}$ is independent of the natural labeling. 
\end{remark}

Theorem \ref{thm:path_to_poset} and Lemma~\ref{lem:poset_to_path} together allow one to draw mutually enriching connections between $P$-Eulerian polynomials and non-nesting rook polynomials. In what follows, we gather corollaries of this correspondence. 

To begin, we note that the ultra-log-concavity consequence of the Neggers--Stanley conjecture is true for naturally labeled width two posets. We can deduce an analogous result for the polynomial $E_{P}(t) = \sum_{j=1}^{|P|}e_{j}(P)t^{j}$ where $e_{j}(P)$ denotes the number of surjective order--preserving maps from $P$ to $[j]$. An alternative formulation of the Neggers--Stanley conjecture for naturally labeled posets $P$ asserts the real-rootedness of $E_{P}$.   

\begin{corollary}\label{neggers_stanley_width2}
Let $P$ be a naturally labeled poset of width two and $\lambda / \mu$ be the skew shape obtained from Lemma~\ref{lem:poset_to_path}. Then the $W$-polynomial of $P$ equals the non-nesting rook polynomial of $\lambda / \mu$: \begin{equation}\label{eq:w_poly_equals_nn_rook}
 W_{P}(t) =  M_{\lambda / \mu}(t),
\end{equation}
and hence $W_{P}$ is ultra-log-concave. Consequently, $E_{P}$ is also ultra-log-concave. 
\end{corollary}

\begin{proof}
Since every natural labeling of the poset $P$ gives rise to the same $W_{P}$, assume $P$ has the canonical labeling $\omega$. By Theorem~\ref{thm:path_to_poset}, the polynomial $\widetilde{W}_{P, \omega}$ is equal, up to a relabeling of the variables, to the basis-generating polynomial of the rook matroid on $\lambda/\mu$. The latter polynomial is Lorentzian~\cite[Theorem 3.10]{BrandenHuh2020Lorentzian}, and this property is  preserved under a linear change of variables \cite[Theorem 2.10]{BrandenHuh2020Lorentzian}. The aforementioned relabeling maps $\xvec$ variables (resp. $\yvec$ variables) labeled by elements of $C_{1}$ (resp. $C_{2}$) to $\xvec$ variables labeled by rows $1, \ldots , r$ (resp. columns $r+1, \ldots , r+c$) of the skew shape. Thus, by specializing the $\xvec$ variables to $1$ and the $\yvec$ variables to $t$ in~(\ref{eq:rook_as_w_polynomial}), we obtain (\ref{eq:w_poly_equals_nn_rook}). The ultra-log-concavity statement of $W_{P}$ then follows from Corollary~\ref{ultra_log_concavity_ncm}. Finally, the ultra-log-concavity of $E_{P}$ follows from \cite[Lemma 2.5 (ii)]{BrandenJochemkoFerroni2024Preservation}.
\end{proof}

In particular, the $P$-Eulerian polynomial of Stembridge's counterexample to the Neggers--Stanley conjecture is ultra-log-concave, despite not being real-rooted. Figure~\ref{fig:stembridge} shows Stembridge's counterexample and the associated skew shape; the $P$-Eulerian polynomial is computed in Corollary~\ref{cor:nn_with_non_real_roots}.

A simple but immediate consequence of this correspondence is that $M_{\lambda / \mu}$ can be interpreted as the $h^{*}$-polynomial of an order polytope. The context for this is a conjecture on the unimodality of the $h^{*}$-polynomial of lattice polytopes with the integral decomposition property~\cite{SchepersVanLangenhoven2013UnimodalityIntegrallyClosed}; see~\cite{Braun2016UnimodalitySurvey} for a survey on unimodality problems in Ehrhart theory.

\begin{corollary}\label{M_lambda_mu_as_h_star}
    Let $\lambda / \mu$ be a skew shape, $P$ be the corresponding naturally labeled poset obtained from Theorem~\ref{thm:path_to_poset}, and $\mathcal{O}(P)$ be the order polytope of $P$. Then \[
    M_{\lambda / \mu}(t) =  h^{*}_{\mathcal{O}(P)}(t).
    \]
    In particular, the $h^{*}$-polynomial of the order polytope of a naturally labeled width two poset is ultra-log-concave.
\end{corollary}

\begin{proof}
    The proof is immediate from Theorem~\ref{thm:path_to_poset} together with the well-known fact that for a naturally labeled poset $P$, its $P$-Eulerian polynomial coincides with the $h^{*}$-polynomial of its order polytope; see, for example,~\cite[Theorem 6.3.11]{BeckSanyal2018CRT}.
\end{proof}

In the other direction of the poset--skew shape correspondence, one can use known results on $P$-Eulerian polynomials to deduce distributional properties of the non-nesting rook numbers, extending the results in Section~\ref{section:distributional_properties}. We mention two such applications below.

Stembridge's width two counterexample to the Neggers--Stanley conjecture \cite{Stembridge2007NeggersStanleyCounter} shows that real-rootedness of $M_{\lambda / \mu}$ can fail in general. This counterexample is illustrated in Figure~\ref{fig:stembridge} together with the corresponding skew shape. The following corollary gives a rook-theoretic and matroid-theoretic context within which Stembridge's counterexample can be understood. 

\begin{corollary}\label{cor:nn_with_non_real_roots}
    The non-nesting rook polynomial $M_{\lambda /\mu}$ is not real-rooted in general. Concretely, $M_{\lambda/\mu}$ is not real-rooted for the skew shape $\lambda / \mu = 888888765/76654321$.
\end{corollary}

\begin{proof}
Let $\lambda / \mu = 888888765/76654321$. Applying Corollary~\ref{neggers_stanley_width2}, we obtain the poset on the left in Figure~\ref{fig:stembridge}, which we recognize as Stembridge's counterexample to Neggers' formulation of the Poset conjecture. Thus $M_{\lambda/\mu}$ is not real-rooted. 

Alternatively, using SageMath~\cite{SageMath}, one can compute $M_{\lambda / \mu}(t)$ to be: \[
M_{\lambda / \mu}(t) = 3t^8 + 86t^7 + 658t^6 + 1946t^5 + 2534t^4 + 1420t^3 + 336t^2 + 32t + 1.
\]
This polynomial has non-real zeros near $-1.85884 \pm 0.14976i$.
\end{proof}
\begin{remark}
    By the experimental evidence gathered in \cite{Stembridge2007NeggersStanleyCounter} and the poset--skew shape correspondence above, it follows that $\lambda/\mu = 888888765/76654321$ is the \emph{smallest} skew shape (by size) for which $M_{\lambda / \mu}$ fails to be real-rooted. It would be interesting to find a large class of skew shapes for which real-rootedness holds. 
\end{remark}

\begin{figure}[htbp]
\begin{subfigure}[b]{0.40 \textwidth}
\centering
\scalebox{0.9}{
\begin{tikzpicture}
\node[dot=4pt, fill=black, label={[xshift=-0.55cm, yshift=-0.25cm, text=blue]{$1$}}] at (0, 0) (a1) {};
\node[dot=4pt, fill=black, label={[xshift=-0.55cm, yshift=-0.25cm, text=blue]{$3$}}] at (0, 2) (a2) {};
\node[dot=4pt, fill=black, label={[xshift=-0.55cm, yshift=-0.25cm, text=blue]{$5$}}] at (0, 4) (a3) {};
\node[dot=4pt, fill=black, label={[xshift=-0.55cm, yshift=-0.25cm, text=blue]{$7$}}] at (0, 6) (a4) {};
\node[dot=4pt, fill=black, label={[xshift=-0.55cm, yshift=-0.25cm, text=blue]{$9$}}] at (0, 8) (a5) {};
\node[dot=4pt, fill=black, label={[xshift=-0.55cm, yshift=-0.25cm, text=blue]{$11$}}] at (0, 10) (z6) {};
\node[dot=4pt, fill=black, label={[xshift=-0.55cm, yshift=-0.25cm, text=blue]{$13$}}] at (0, 12) (z7) {};
\node[dot=4pt, fill=black, label={[xshift=-0.55cm, yshift=-0.25cm, text=blue]{$14$}}] at (0, 14) (z8) {};
\node[dot=4pt, fill=black, label={[xshift=-0.55cm, yshift=-0.25cm, text=blue]{$16$}}] at (0, 16) (z9) {};

\node[dot=4pt, fill=black, label={[xshift=0.55cm, yshift=-0.25cm, text=blue]{$2$}}] at (4, 0) (a6) {};
\node[dot=4pt, fill=black, label={[xshift=0.55cm, yshift=-0.25cm, text=blue]{$4$}}] at (4, 2) (a7) {};
\node[dot=4pt, fill=black, label={[xshift=0.55cm, yshift=-0.25cm, text=blue]{$6$}}] at (4, 4) (a8) {};
\node[dot=4pt, fill=black, label={[xshift=0.55cm, yshift=-0.25cm, text=blue]{$8$}}] at (4, 6) (a9) {};
\node[dot=4pt, fill=black, label={[xshift=0.55cm, yshift=-0.25cm, text=blue]{$10$}}] at (4, 8) (a10) {};
\node[dot=4pt, fill=black, label={[xshift=0.55cm, yshift=-0.25cm, text=blue]{$12$}}] at (4, 10) (y11) {};
\node[dot=4pt, fill=black, label={[xshift=0.55cm, yshift=-0.25cm, text=blue]{$15$}}] at (4, 12) (y12) {};
\node[dot=4pt, fill=black, label={[xshift=0.55cm, yshift=-0.25cm, text=blue]{$17$}}] at (4, 14) (y13) {};

\draw[thick,black] (a1)-- (a2) -- (a3)-- (a4)-- (a5) -- (z6) -- (z7) -- (z8) -- (z9);
\draw[thick,black] (a6)-- (a7)-- (a8)-- (a9)-- (a10)--(y11)--(y12)--(y13);
\draw[thick, black] (a1)-- (y11);
\draw[thick,black] (a2)-- (y12);

\draw[thick,black] (a3)-- (y13);

\draw[thick,black] (y12)-- (z9);
\draw[thick,black] (a2)-- (a6);
\draw[thick,black] (a3)-- (a7);
\draw[thick,black] (a4)-- (a8);
\draw[thick,black] (a5)-- (a9);
\draw[thick,black] (z6)-- (a10);
\draw[thick,black] (z7)-- (y11);

\end{tikzpicture}
}
\end{subfigure}
~
\hspace{0.5cm}
\begin{subfigure}[b]{0.40\textwidth}
\centering
\scalebox{0.9}{
\begin{tikzpicture}[inner sep=0in,outer sep=0in]
\node (n) {\begin{varwidth}{6cm}{
\ytableausetup{boxsize=1.35em}
 \begin{ytableau} \none & \none[2] & \none[4] & \none[6] & \none[8] & \none[10] & \none[12] & \none[15] & \none[17]  \\ \none[16] & \none & \none & \none & \none & \none & \none & \none &  \\ \none[14] & \none & \none & \none & \none & \none & \none &  &  \\ \none[13] & \none & \none & \none & \none & \none & \none &  &  \\ \none[11] & \none & \none & \none & \none & \none &  &  &  \\ \none[9] & \none & \none & \none & \none &  &  &  &  \\ \none[7] & \none & \none & \none &  &  &  &  &  \\ \none[5] & \none & \none &  &  &  &  &  \\ \none[3] & \none &  &  &  &  &  \\ \none[1] &  &  &  &  &  \\ \end{ytableau}}
\end{varwidth}};

\end{tikzpicture}
}
\end{subfigure}
\caption{On the left: Stembridge's  counterexample~\cite{Stembridge2007NeggersStanleyCounter} to the Neggers--Stanley conjecture (albeit with a different natural labeling). On the right: the corresponding skew shape, with the poset labeling of rows and columns.} 
\label{fig:stembridge}
\end{figure}

Using the skew shape -- poset correspondence, we can also deduce precisely when the non-nesting rook polynomial is \defin{palindromic}, i.e. when its coefficient sequence $(a_{k})_{k=0}^{n}$ satisfies $a_{k} = a_{n-k}$, for $k =0, \ldots , n$. A palindromic polynomial $A(t)$ can always be expressed in the following basis:
\[
A(t) = \sum_{k = 0}^{\lfloor\frac{n}{2}\rfloor}\gamma_{k}t^{k}(1+t)^{n-2k}
\]
When the coefficients satisfy $\gamma_{k} \geq 0$ for all $k$, $A(t)$ is called \defin{$\gamma$-positive}; this property is known to imply unimodality of $A(t)$. 

The skew shapes that satisfy symmetry of the non-nesting rook polynomial are defined below.

\begin{definition}
    A skew shape $\lambda/\mu$ is a \defin{squarecase} if the bounding rectangle of $\lambda/\mu$ is a square and all inner and outer corners of $\lambda/\mu$ are of the form $(i, r+i)$. 
\end{definition}

For such shapes, $\gamma$-positivity --- and hence unimodality --- holds for the non-nesting rook polynomial. This fact is not immediately obvious from the combinatorics of rook placements alone. In the proof below, the direct sum of skew shapes is an operation defined in the proof of Lemma~\ref{lem:direct_sum}. 

\begin{corollary}\label{cor:nn_gamma_positive}
    Let $\lambda / \mu$ be a skew shape. Then $M_{\lambda / \mu}$ is palindromic if and only if $\lambda/\mu$ is the direct sum of squarecase skew shapes. When this holds, $M_{\lambda / \mu}$ is $\gamma$-positive. 
\end{corollary}

\begin{proof}
It suffices to show that the result holds for connected skew shapes. This is because for skew shapes $D_{1}, D_{2}$ the polynomial $M_{D_{1}\oplus D_{2}} = M_{D_{1}}\cdot M_{D_{2}}$ is palindromic if and only if each $M_{D_{i}}$ is palindromic for $i=1,2$.

By the poset -- skew shape correspondence and the well-known criterion for the palindromicity of $P$-Eulerian polynomials~\cite[Corollary 3.15.18]{StanleyEC1}, $M_{\lambda/\mu}$ is palindromic if and only if the corresponding width two poset is graded. By unwinding what the graded condition means for the skew shape, we find that it holds if and only if the inner and outer corners of $\lambda/\mu$ are of the form $(i, r+i)$, where $r$ is the number of rows of the skew shape. The $\gamma$-positivity of $M_{\lambda / \mu}$ then follows from Theorem~\ref{thm:path_to_poset} and \cite[Theorem 4.2]{Branden06SignGraded}.
\end{proof}

\section{Funding}

This work was supported by the Wallenberg AI, Autonomous Systems and Software Program (WASP) funded by the Knut and Alice Wallenberg Foundation. 

\section{Acknowledgements}
We thank Jörgen Backelin and Maena Quemener for initial contributions to this project and helpful discussions throughout. We thank Joe Bonin for spotting an error in an earlier version of this paper, and are grateful to him for other clarifications and thoughtful correspondence. Finally, we are very grateful to Petter Bränd\'{e}n and Katharina Jochemko for their feedback and several stimulating conversations, as well as Benjamin Schröter for valuable insights. 
\bibliographystyle{alpha}
\bibliography{bibliography}

\end{document}